\documentstyle[11pt,amssymb,amsmath,amscd,amsbsy,amsfonts,theorem,hyperref,accents]{article}

\input xy
\xyoption{all}

\newcommand{\zz}{{\Bbb Z}}

\newcommand{\nn}{{\Bbb N}}
\newcommand{\cc}{{\Bbb C}}
\newcommand{\rr}{{\Bbb R}}
\newcommand{\qq}{{\Bbb Q}}
\newcommand{\pp}{{\Bbb P}}

\newcommand{\ff}{{\Bbb F}}

\newcommand{\ddim}{\operatorname{dim}}
\newcommand{\ddeg}{\operatorname{deg}}
\newcommand{\kker}{\operatorname{Ker}}

\newcommand{\Homi}{\underline{\operatorname{Hom}}}
\newcommand{\Hom}{\operatorname{Hom}}
\newcommand{\End}{\operatorname{End}}
\newcommand{\Mat}{\operatorname{Mat}}
\newcommand{\op}[1]{\operatorname{#1}}
\newcommand{\kbar}{\overline{k}}

\newcommand{\ffi}{\varphi}

\newcommand{\eps}{\varepsilon}

\newcommand{\la}{\langle}
\newcommand{\ra}{\rangle}
\newcommand{\row}{\rightarrow}

\newcommand{\low}{\leftarrow}
\newcommand{\lrow}{\longrightarrow}
\renewcommand{\leq}{\leqslant}
\renewcommand{\geq}{\geqslant}

\newcommand{\calm}{{\cal M}}
\newcommand{\hm}{\operatorname{H}_{\calm}}
\newcommand{\nichego}[1]{}

\newcommand{\ov}[1]{\overline{#1}}
\newcommand{\un}[1]{\underline{#1}}
\newcommand{\wt}[1]{\widetilde{#1}}

\newcommand{\smk}{{\mathbf{Sm}}_k}

\newcommand{\sln}[1]{\op{S}_{L-N}^{#1}}
\newcommand{\slnT}{\op{S}_{L-N}^{Tot}}
\newcommand{\tslnT}{\wt{\op{S}}_{L-N}^{Tot}}

\newcommand{\laz}{{\Bbb L}}
\newcommand{\cd}{{\cal D}}

\newcommand{\ck}{{\cal K}}

\newcommand{\ct}{{\cal T}}

\newcommand{\CH}{\operatorname{CH}}

\newcommand{\bub}{*=0{\bullet}}
\newcommand{\dub}{.}
\newcommand{\wub}{*=0{\circ}}

\newcommand{\shk}{{\cal SH}_{{\Bbb{A}}^1}(k)}

\newcommand{\dmk}{\op{DM}(k)}

\newcommand{\dmgmk}{\op{DM}_{gm}(k)}

\newcommand{\dmkF}[1]{\op{DM}(k;#1)}

\newcommand{\dmISF}[2]{\op{DM}({#1}/{#1};{#2})}
\newcommand{\dmEF}[2]{\op{DM}({#1};{#2})}

\newcommand{\km}{k_M}
\newcommand{\moco}[2]{\op{H}_{{\cal M}}^{#1,#2}}
\newcommand{\mocod}[1]{\op{H}_{{\cal M}}^{\{#1\}}}
\newcommand{\moho}[2]{\op{H}^{{\cal M}}_{#1,#2}}
\newcommand{\mohod}[1]{\op{H}^{{\cal M}}_{\{#1\}}}

\newcommand{\HH}{\op{H}}

\newcommand{\Qed}{\hfill$\square$\smallskip}
\newcommand{\Red}{\hfill$\triangle$\smallskip}

\newenvironment{proof}{\noindent{\it Proof}:}{\vskip 5mm}

\newtheorem{prop}{Proposition}[section]{\bf}{\it}
\newtheorem{thm}[prop]{Theorem}{\bf}{\it}
\newtheorem{lem}[prop]{Lemma}{\bf}{\it}
{\bf}{\it}
\newtheorem{defi}[prop]{Definition}{\bf}{\it}
\newtheorem{conj}[prop]{Conjecture}{\bf}{\it}
\newtheorem{observ}[prop]{Observation}{\bf}{\it}
\newtheorem{exa}[prop]{Example}{\bf}{\it}
\newtheorem{rem}[prop]{Remark}{\bf}{}
{\bf}{\it}

{\bf}{\it}
{\bf}{\it}
\newtheorem{cor}[prop]{Corollary}{\bf}{\it}
{\bf}{\it}

\begin{document}

\title{Torsion motives}
\author{Alexander Vishik}
\date{}
\maketitle


\begin{abstract}
 In this paper we study Chow motives whose identity map is killed by a natural number. Examples of such objects were
 constructed by Gorchinskiy-Orlov \cite{GO}. We introduce various invariants of torsion motives, in particular, the
 {\it $p$-level}. We show that this invariant bounds from below the dimension of the variety a torsion motive $M$ is a direct summand of and imposes restrictions on motivic and singular cohomology of $M$. 
 We study in more details the $p$-torsion motives of surfaces,
 in particular, the Godeaux torsion motive. We show that
 such motives are in $1$-to-$1$ correspondence with certain
 Rost cycle submodules of free modules over $H^*_{et}$.
 This description is parallel to that of mod-$p$ reduced
 motives of curves. 
\end{abstract}

\section{Introduction}

The purpose of this article is to initiate the investigation
of an interisting class of Chow motives which, for some reason, didn't attract attention they deserve before. These
are the so-called {\it torsion motives}, that is, Chow 
motives which disappear with rational coefficients. 

For me, the motivation for studying such objects comes from
{\it isotropic realisation functors} of \cite{IM}.
Such realizations $\psi_E:\dmkF{\ff_p}\row\dmISF{\wt{E}}{\ff_p}$
are parameterized by the finitely-generated field extensions
$E/k$ of the ground field and
take values in the {\it isotropic motivic categories}.
Here $\wt{E}=E(t_1,t_2,\ldots)$ is the {\it flexible closure}
of $E$, and isotropic motivic category
$\dmISF{F}{\ff_p}$ is obtained from $\dmEF{F}{\ff_p}$ by
moding out the motives of $p$-anisotropic varieties over $F$.
It appears that, for {\it flexible} fields \cite[1.2]{IM}, the isotropic
motivic categories are sufficiently small, with Hom groups
between compact objects expected to be finite. This permits
to assign to an object of the ``global'' Voevodsky motivic 
category
with finite coefficients its ``local'' much simpler versions. 
At the same time, this collection of isotropic realization 
functors (together with the, so-called, {\it thick} versions
of them - see \cite[Sect. 5]{IM}) is not entirely 
conservative. And the expectation is that the kernel of 
this family on motives with $\zz_p$-coefficients 
is generated exactly by {\it torsion motives}.
Thus, understanding them is essential for completing the 
``isotropic picture''. 

It is not a'priori obvious that torsion motives should
even exist. But examples of such objects were constructed by
Gorchinskiy-Orlov \cite{GO} as direct summands in the motives
of Godeaux, Beauville and Burniat surfaces (based on earlier
works by Alexeev-Orlov \cite{AO}, B\"{o}hning-Graf von Bothmer-Sosna \cite{BGS} and Galkin-Shinder \cite{GS}).
Gorchiskiy-Orlov used these motives to construct the first
known example of the {\it phantom category}.

We start our investigation by assigning certain invariants
to torsion motives, which measure the complexity of such objects. Namely, by the results of \cite{VY}, any Chow 
motive $M$ can be lifted to a cobordism motive $M^{\Omega}$
and so (by the
universality of algebraic cobordism of Levine-Morel -\cite{LM}, we assume $char(k)=0$), to a motive $M^A$ in the sense of an arbitrary
oriented theory $A^*$. The annihilator of the cobordism
version provides an ideal in the Lazard ring, whose radical
is the intersection of finitely many ideals $I(p,n)$,
where $p$ is prime and $I(p,n)=(p,v_1,\ldots,v_{n-1})$
are invariant prime ideals of Landweber \cite{La73b}
($v_i$ is a $\nu_i$-element of dimension $p^i-1$). 
This leads to the notion
of a {\it $p$-level} of a motive - see 
Definition \ref{level} -
the number $n$ appearing above ($\infty$, if $p$ is not a torsion prime, and zero, if $M$ is not a torsion motive). This invariant can be
interpreted in terms of Morava K-theory versions of $M$.
Namely, if $M$ has $p$-level $n$, then $M^{K(r)}=0$, for
$0\leq r<n$, and $M^{K(r)}\neq 0$, for $r\geq n$ - see 
Corollary \ref{In-Morava}. Since $K(1)$ is just the usual
K-theory, localized at $p$ and re-oriented (in the sense of \cite{Qu71}, or \cite{PS}), it follows that
the motives of $p$-level $>1$ are exactly the {\it K-phantom motives}, that is, such Chow motives, whose K-version
is trivial. The level appears to be a rather rigid invariant - it is stable in tensor powers: the $p$-level of $M$ is the same as that of $M^{\otimes r}$. 

The $p$-level imposes various restrictions on the motive. In particular, on the dimension of the respective smooth projective variety $X$ (whose motive $M$ is a direct summand of).
In Theorem \ref{level-size} we show that, for a motive
of $p$-level $n$, such dimension is 
$\displaystyle\geq\frac{p^{n}-1}{p-1}$. Moreover, over
any field extension $F/k$, our motive has no Chow groups
in co-dimensions smaller than the specified bound - see 
Proposition \ref{CH-small}. These results are obtained 
using {\it Symmetric operations} of \cite{SOpSt}.
We also show that if, for a motive of $p$-level $n$, the $\displaystyle\ddim(X)\leq
\frac{(2p-1)(p^{n-1}-1)}{p-1}$, then the associated ideal
$\ov{J}(M)$ (Definition \ref{JM}) is radical. This implies -
Corollary \ref{torsion-M-Margolis} that Milnor's operations
$Q_i$, $0\leq i< n$ of Voevodsky act as exact 
differentials on the motivic cohomology of $M$ (that is, $\op{Im}(Q_i)=\op{Ker}(Q_i)$). 
This shows, in particular, that over an algebraically closed
field, the etale cohomology of $M$ (with finite coefficients) 
are absent in dimensions and co-dimensions $<n$. That is,
the motive $M$ can't reside ``too close to the surface''
of the motive of the respective variety, which suggests that
the bound on the possible dimension of $X$ may be improved.
This is indeed done in Corollary \ref{sharper-m-tri}: for
$n\leq 4$ and algebraically closed ground field,
$\displaystyle\ddim(X)\geq 
 \lceil\frac{n}{2}\rceil+\frac{p^{n}-1}{p-1}$. We conjecture that the same holds in general. 
 
 In the second part of the paper we take a more detailed look at
 the torsion motives of surfaces. We start with the torsion direct summand in the motive of the Godeaux surface. 
 We show that (as for any torsion motive of a surface), such
 a motive $M$ is completely described by the zero-slice of it in
 the homotopy $t$-structure \cite{Deg}. This slice corresponds to the Rost cycle module $A$ which is naturally
 a submodule of a free module over $\HH^*_{et}=\HH^*_{et}(-,\zz/5)$ of rank two with generators in degrees $-1$ and $-2$. The component
 of degree zero of $A$ is exactly $\CH_0(M)$ - the group of
 zero-cycles on $M$ (considered over all field extensions
 $F/\cc$). The latter group can be described as the group
 of $\HH^*_{et}(F)$-relations between two unramified elements
 $u\in \HH^1_{nr}(\cc(X)/\cc,\zz/5)$ and
 $v\in \HH^2_{nr}(\cc(X)/\cc,\zz/5)$ in $\HH^3_{et}(F(X),\zz/5)$.
 Among such relations there is the generic one: $(v,u)$, corresponding to $F=\cc(X)$. We prove that this element of $A^0$ generates $A$ as a Rost cycle module - Proposition
 \ref{A-gen}. The mentioned
 free module $\HH^*_{et}\la-1\ra\oplus \HH^*_{et}\la -2\ra$
 possesses a natural bilinear form (which is an extension
 of the standard hyperbolic form on $\ff_5\la -1\ra\oplus
 \ff_5\la -2\ra$ into the skew-commutative realm of 
 $\HH^*_{et}$). We show that our submodule $A$ is {\it Lagrangian} with respect to this form - 
 Proposition \ref{A-max-ort}. Finally, we express
 the diagonal class of the Godeaux surface in terms of the
 classes $u$ and $v$.
 
 We prove that torsion motives of surfaces satisfy the Krull-Schmidt principle - see Theorem \ref{K-S}. We look specifically at $p$-torsion motives over an algebraically closed field. 
 Assigning to $M$ the zero-th $t$-homotopic slices of $M$
 and $M^{\vee}$ we obtain a pair $A$ and $B$ of Rost cycle modules which are naturally submodules of free $\HH^*_{et}$ cycle modules $H_A$ and $H_B$ with generators in degrees
 $-1$ and $-2$. There is a natural ($\HH^*_{et}$-valued) pairing between $H_A$ and $H_B$, such that $A$ and $B$
 are orthogonal to each other with respect to it. 
 As in the case of the Godeaux motive above, we can describe
 the generators of $A$ and $B$ - Proposition 
 \ref{generators-A}. In particular, this gives the description of $\CH_0(M)$. The above assignment provides a
 $1$-to-$1$ correspondence between $p$-torsion direct summands in the motives of surfaces and pairs of submodules of dual free
 $\HH^*_{et}$-modules with generators in degrees $-1$ and $-2$, satisfying certain conditions - see Theorem 
 \ref{tor-sur-oto}. This description is completely parallel to the description of direct summands
 of the reduced motives of curves with $\zz/p$-coefficients
 (Theorem \ref{curv}). The only difference is that, in the 
 latter case, the respective dual free modules have generators in degree $-1$ only (and duality is shifted accordingly). The abundance of mod-$p$ motives of curves
 suggests that a similar situation should take place in the 
 case of $p$-torsion motives of surfaces. 
 We show that any non-zero $p$-torsion motive of surfaces has a non-trivial Picard group $\CH^1(M_{\kbar})$ over an algebraic closure 
 - Corollary \ref{Pic-nontriv}. 
 Finally, we describe the automorphism group of a torsion
 motive in terms of the respective embedding 
 $A\subset H_A$ - Proposition \ref{RA-Mat-nilp}.

 \medskip

\noindent
{\bf Acknowledgements:}
I would like to thank Mikhail Bondarko and Burt Totaro for useful comments.
The support of the EPSRC standard grant EP/T012625/1 is gratefully acknowledged.
Finally, I would like to thank the Referees for many useful remarks which improved the text.

\section{Torsion motives and their invariants}
\label{section-two}

Throughout the article we will consider only fields of characteristic zero (the exception is Subsection \ref{subsection-three-one}, where results hold over an arbitrary field).  
In Sections \ref{section-two}, \ref{section-three} (aside from \ref{subsection-three-one}) and Section \ref{section-five} our ground field will be any field of characteristic zero.
In Section \ref{section-four} we will work mostly over an algebraically closed field of characteristic zero, respectively $\cc$, but some results will still be proven for an arbitrary field of characteristic zero.

Recall that to any oriented cohomology theory one can assign
a formal group law. Among such theories there is a universal
one - the algebraic cobordism of Levine-Morel $\Omega^*$
\cite[Theorem 1.2.6]{LM}. The formal group law, in this case,
is also universal. In particular, 
$\Omega^*(\op{Spec}(k))=\laz$ - the Lazard ring. To any oriented cohomology theory $A^*$, one can assign its graded version $GrA^*$ corresponding to the
co-dimension of support filtration $F_r$ on $A^*$, where an element belongs to $F_r$, if it vanishes on a complement to some closed subscheme
of co-dimension $r$. In particular, below we will use such a graded version
$Gr\Omega^*$ for the algebraic cobordism. We have the natural surjection
of $\laz$-algebras $\phi:\CH^*\otimes_{\zz}\laz\twoheadrightarrow Gr\Omega^*$ - \cite[Corollary 4.5.8]{LM} which is an isomorphism rationally.

To any oriented cohomology theory $A^*$ one can assign the category
of $A^*$-Chow motives $Chow^A(k)$ - see, for example, \cite[Sect. 2]{VY}. Any morphism
of oriented cohomology theories $A^*\row B^*$ gives the functor
$Chow^A(k)\row Chow^B(k)$.  In particular, the natural projection $\Omega^*\row\CH^*$ from the algebraic cobordism of Levine-Morel to Chow groups leads to the functor $Chow^{\Omega}(k)\row Chow^{CH}(k)=Chow(k)$.

By \cite[Corollary 2.8]{VY}, the mentioned functor defines a 1-to-1 correspondence between isomorphism classes of Chow-motives and $\Omega^*$-motives. Let $M$ be an object of $Chow(k)$ (with integral coefficients) and $M^{\Omega}\in Ob(Chow^{\Omega}(k))$ be the respective $\Omega^*$-motive. This permits to assign the following invariant to $M$.

\begin{defi}
 \label{JM}
 Let $M\in Ob(Chow(k))$. Define $J(M)\vartriangleleft\laz$
 as the annihilator of $id_{M^{\Omega}}$. In other words, it is the annihilator of the projector $\rho^{\Omega}$ in $\Omega^*(X^{\times 2})$, where $M^{\Omega}=(X,\rho^{\Omega})$.\\
 Define $\ov{J}(M)\vartriangleleft\laz$ as the annihilator of 
 $\rho^{\Omega}$ in $Gr\Omega^*(X^{\times 2})$. 
\end{defi}

Obviously, $J(M)\subset\ov{J}(M)$.
Recall that on the algebraic cobordism of Levine-Morel 
$\Omega^*$ we have the action of Landweber-Novikov operations
\cite[Example 4.1.25]{LM}. The total Landweber-Novikov operation
$\slnT:\Omega^*\row\Omega^*[b_1,b_2,\ldots]$ is multiplicative (respects the product). These operations descend to the graded algebraic cobordism.
Moreover, the action on $Gr\Omega^*$ is much simpler. Namely, if $z\in\CH(X)$, $u\in\laz$ and we denote $\phi(z\otimes u)$ as 
$z\cdot u\in Gr\Omega^*(X)$, then $\slnT(z\cdot u)=z\cdot\slnT(u)$.
Considering $z=\rho\in\CH_{\ddim(X)}(X^{\times 2})$, we obtain:

\begin{observ}
 \label{ovJ-LN}
 The ideal $\ov{J}(M)$ is stable under Landweber-Novikov operations.
\end{observ}

\begin{rem}
Clearly, under field extensions, these ideals can only increase. So, these invariants are trivial for geometrically split motives, and even for motives containing any split components over some extensions. In particular, for (the whole) motives of varieties. 
\Red
\end{rem}

\begin{defi}
 \label{TM}
 A Chow motive $M$ is a 'torsion motive', if there exists $n\in\nn$ such that $n\cdot id_M=0$.
\end{defi}

\begin{prop}
 \label{torsion-JM}
 The following conditions are equivalent: 
 \begin{itemize}
  \item[$(1)$] $M$ is a torsion motive; 
  \item[$(2)$] $J(M)\neq 0$.
 \end{itemize}
\end{prop}

\begin{proof} Suppose $M=(X,\rho)$.
 $(1)\row (2)$: If $n\cdot\rho=0$, then $n\cdot\rho^{\Omega}$ has support in codimension $>\ddim(X)$. Hence, it is $\circ$-nilpotent. But $\rho^{\Omega}$ is a projector. Hence, some
 power of $n$ belongs to $J(M)$.
 
 $(2)\row(1)$: Recall, that the total Landweber-Novikov operation acts on $Gr\Omega^*(X^{\times 2})$ as follows: for $z\in\CH(X^{\times 2})$ and 
 $u\in\laz$, we have:
 $\slnT(z\cdot u)=z\cdot\slnT(u)$.
 So, if $u\in J(M)$, then for any individual Landweber-Novikov operation $\sln{\ov{b}}$, the product
 $\sln{\ov{b}}(u)\cdot\rho^{\Omega}$ has support in co-dimension $>\ddim(X)$, so is nilpotent. Thus, some power
 of $\sln{\ov{b}}(u)$ belongs to $J(M)$. But $u\neq 0\Leftrightarrow (\slnT(u))_{deg=\ddim(u)}\neq 0$ and the latter elements reside in $\laz_{\ddim=0}=\zz$, which has
 no nilpotents. Hence, if $J(M)\neq 0$, it contains some non-zero integers.
 \Qed
\end{proof}

We also can bound $J(M)$ by an invariant ideal from below.

\begin{defi}
 \label{JMtilde}
 Let $M\in Ob(Chow(k))$. Define 
 $\wt{J}(M)\vartriangleleft\laz$
 as the annihilator of $\slnT(\rho^{\Omega})$ in $\Omega^*(X^{\times 2})$, where $M^{\Omega}=(X,\rho^{\Omega})$.
\end{defi}

\begin{prop}
 \label{correct-JMt}
 The ideal $\wt{J}(M)$ is well-defined, that is, it does not depend on the presentation of $M$ in the form $(X,\rho)$.
\end{prop}

\begin{proof}
 For an $\Omega^*$-correspondence 
 $\alpha:X\rightsquigarrow Y$, let us introduce
 $\tslnT(\alpha):=\slnT(\alpha)\cdot\op{S}^{L-N}_{Tot}(1_Y)$.
 It follows from the Riemann-Roch theorem (for multiplicative operations) - Panin \cite[Theorem 2.5.4]{P-RR} that
 it respects composition of correspondences:
 $\tslnT(\beta\circ\alpha)=\tslnT(\beta)\circ\tslnT(\alpha)$.
 
 Let $(X,\rho)$ and $(Y,\eps)$ be two presentations for $M^{\Omega}$ (from now on we will drop the superscript $\Omega$ in the projectors). Then there are $\Omega^*$-correspondences
 $f:X\rightsquigarrow Y$ and $g:Y\rightsquigarrow X$ such
 that $\rho=g\circ\eps\circ f$ and $\eps=f\circ\rho\circ g$, and so, $\tslnT(\rho)=\tslnT(g)\circ\tslnT(\eps)\circ\tslnT(f)$, and similarly for $\eps$. Hence, $u\in\laz$ annihilates $\tslnT(\eps)$ if and only if it annihilates
 $\tslnT(\rho)$. But $\op{S}^{L-N}_{Tot}(1_Z)$ is invertible, so $u\cdot\slnT(\eps)=0\Leftrightarrow u\cdot\slnT(\rho)=0$.
 \Qed
\end{proof}

\begin{prop}
 \label{JMt-LNinv}
 The ideal $\wt{J}(M)$ is invariant under Landweber-Novikov operations.
\end{prop}

\begin{proof}
 The total Landweber-Novikov operation is multiplicative, so: $\slnT(x\cdot u)=\slnT(x)\cdot\slnT(u)$.
 This implies by induction on the degree of the monomial
 $\ov{b}^{\ov{r}}$ that
 $\slnT(\rho)\cdot\sln{\ov{b}^{\ov{r}}}(u)=0$, if
 $\slnT(\rho)\cdot u=0$. 
 \phantom{a}\hspace{5mm}
 \Qed
\end{proof}

The following result shows that the ideals $J(M)$, $\ov{J}(M)$ and $\wt{J}(M)$ have the same radical.

\begin{prop}
 \label{JM-JMt}
 $\wt{J}(M)\subset J(M)\subset\ov{J}(M)\subset\sqrt{\wt{J}(M)}$.
\end{prop}

\begin{proof}
 By definition, $\wt{J}(M)\subset J(M)\subset\ov{J}(M)$.
 
 Suppose $u\in J(M)$. Let us show by induction on $r$ that the degree $r$ part $(\slnT(\rho))_r$ of the total Landweber-Novikov operation applied to $\rho$ is annihilated by some power of $u$. The base of induction
 is provided by our condition $\rho\cdot u=0$ on $u$.
 Suppose, $(\slnT(\rho))_{<r}\cdot u^a=0$. Since
 $0=(\slnT(\rho\cdot u))_r=\left((\slnT(\rho))_{<r}\cdot
 (\slnT(u))_{>0}\right)_r+(\slnT(\rho))_r\cdot u$, we obtain that $(\slnT(\rho))_r\cdot u^{a+1}=0$.
 
 If $v\in\ov{J}(M)$, then $v\cdot\rho$ is nilpotent by co-dimensional considerations. But $\rho$ is an idempotent.
 Hence, some power of $v$ annihilates $\rho$ already in $\Omega^*$, and so, belongs to $J(M)$.
 \Qed
\end{proof}

Landweber-Novikov operations provide a rich structure on the Lazard ring. In particular, prime ideals of $\laz$ invariant under these operations can be classified. This was done by Landweber in \cite[Theorem 2.7]{La73b}.
Finitely generated non-zero invariant prime ideals of $\laz$ are exactly ideals
$I(p,n)=(p,v_1,\ldots,v_{n-1})$, where $p$ is prime and $v_i$ is a $\nu_i$-{\it element} of dimension $p^i-1$. Such ideals are parametrized by pairs $(p,n)$, where $p$ is prime and $n$ is a natural number.
In addition, the ideal $I(0)=(0)$ is also prime and invariant.
\begin{prop}
 \label{rJMt-Ipn} For a torsion motive $M$, 
 $$\sqrt{J(M)}=\sqrt{\ov{J}(M)}=\sqrt{\wt{J}(M)}=\operatornamewithlimits{\bigcap}_{(p,n)}I(p,n),
 $$
 where $(p,n)$ runs over finitely many distinct pairs.
\end{prop}

\begin{proof}
 It follows from \cite[proof of Theorem 4.1]{ACMLR} that
 $\ov{J}(M)$ is generated by elements $u$ such that the codimension of $\rho^{\Omega}\cdot u$ is positive. In particular, this ideal is finitely generated.
 By Observation \ref{ovJ-LN}, this ideal of $\laz$ is also invariant under Landweber-Novikov operations. By the Theorem of Landweber - \cite[Proposition 3.4]{La73b},
 $\ov{J}(M)=Q_1\cap\ldots\cap Q_r$, where $P_i=\sqrt{Q_i}$ are invariant finitely generated prime ideals, which by \cite[Theorem 2.7]{La73b} should have the form: $P_i=I(p_i,n_i)$. Then $\displaystyle\sqrt{\ov{J}(M)}=\cap_i P_i=\cap_i I(p_i,n_i)$. Finally, from Proposition \ref{JM-JMt}
 it follows that 
 $\sqrt{J(M)}=\sqrt{\ov{J}(M)}=\sqrt{\wt{J}(M)}$.
 \Qed
\end{proof}

It follows from \cite[Proposition 2.21]{S18b} that
$\wt{J}(M)$ is finitely generated too. Thus, Proposition
\ref{JMt-LNinv} and the mentioned result of Landweber imply
that $\wt{J}(M)=\wt{Q}_1\cap\ldots\cap \wt{Q}_r$ for some
finitely generated invariant primary ideals with the same
radicals $\sqrt{\wt{Q}_i}=P_i$.\\

To simplify the picture, we will move to motives with $\zz_{(p)}$-coefficients. Recall that, as in topology, 
the $p$-localized
algebraic cobordism $\Omega^*_{\zz_{(p)}}$ naturally splits
as a polynomial algebra over the theory $BP^*$ with
the coefficient ring $BP=\zz_{(p)}[v_1,v_2,\ldots]$,
where $v_i$ are $\nu_i$-elements of dimension $p^i-1$.
We will denote the ideal $(p,v_1,\ldots,v_{n-1})$
of $BP$ as $I(n)$. The Landweber-Novikov operations descend 
naturally to the $BP$-theory and the respective prime
invariant finitely generated ideals are exactly $I(n)$, for 
$n\in\nn$, and $I(0)=(0)$ - \cite{La73b}. We will denote the respective Landweber-Novikov algebra as $BP_*BP$ as in topology. In the $p$-localized case we will
use the same notations $J(M)$ and $\ov{J}(M)$ for the similar 
ideals
of $BP$.
We have the graded version $GrBP$ and the natural surjection
$\phi:\CH^*\otimes_{\zz}BP\twoheadrightarrow GrBP^*$ which we
will denote simply by: $\phi(z\otimes u)=:z\cdot u$.
\begin{cor}
 \label{Zp-nur}
 Let $M\in Ob(Chow(k,\zz_{(p)}))$ be non-zero, then $\sqrt{J(M)}=I(n)$, for some $n$. In particular, $J(M)$ contains some powers of $v_r$, for all $r<n$, and doesn't contain any such powers, for $r\geq n$.
\end{cor}

\begin{defi}
 \label{level}
 Let $M\in Ob(Chow(k))$. Define the '$p$-level' of $M$ as the number $n$, such that\\
 $\sqrt{J(M_{\zz_{(p)}})}=I(n)$, and $\infty$, if $M_{\zz_{(p)}}=0$.
\end{defi}

Let $K(r)$, for $r\geq 1$, be the $r$-th Morava K-theory. It is obtained from $\Omega^*(X)$ by change of coefficients:
$K(r)(X)=\Omega^*(X)\otimes_{\laz}K(r)$, where the coefficient ring $K(r)$ is $\zz/p[v_r,v_r^{-1}]$. We can also introduce $K(0)(X)=\CH^*(X)\otimes_{\zz}\qq$.
As was explained above, any Chow motive corresponds uniquely to a cobordism motive. Having a cobordism motive, we can construct $A^*$-motive for any oriented theory $A^*$, by using the {\it universality} of algebraic cobordism $\Omega^*$ - \cite[Theorem 1.2.6]{LM}. In particular, any Chow motive $M$ produces a sequence of Morava-motives
$M^{K(r)}$.
We can interpret the $p$-level of $M$ in terms of these motives.

\begin{prop}
 \label{In-Morava}
 The $p$-level $n$ of $M$ is uniquely determined from the condition: $M^{K(r)}=0$, for $r<n$, and $M^{K(r)}\neq 0$, for $r\geq n$.
\end{prop}

\begin{proof}
 The fact that $M^{K(r)}=0$, for $r<n$, is obvious, as $J(M)$ contains some
 power of $v_r$, while this element is inverted in $K(r)$.
 
 Let now $r\geq n$. Let $M=(X,\rho)$ and $M^{\vee}=(X,\rho^{\vee})$ be the direct summand given by the dual projector (we swap two factors in $X^{\times 2}$). Then no power of $v_r$ annihilates $BP(M\otimes M^{\vee})$ (as the identity map of $M$ is represented by the diagonal class there), and as $M$ is $p$-primary torsion, no power of $v_r$ annihilates the mod-$p$
 reduction $\ov{BP}(M\otimes M^{\vee})$, where $\ov{BP}=BP/p$. 
 Then no power of $v_r$ will annihilate the graded version $Gr\ov{BP}(M\otimes M^{\vee})$. Consider the graded component $\ov{BP}_{(c)}(M\otimes M^{\vee})$ of codimension of support $=c$. We have two cases:
 1) $v_r\in\sqrt{Ann(\ov{BP}_{(c)}(M\otimes M^{\vee}))}$;
 2) $v_r\not\in\sqrt{Ann(\ov{BP}_{(c)}(M\otimes M^{\vee}))}$.
 
 1) From Sechin \cite[Proposition 2.21]{S18b} we know that 
 $\ov{BP}_{(c)}(M\otimes M^{\vee})$ is a union of finitely presented
 $BP_*BP$ modules, which by the result of Landweber \cite[Lemma 3.3]{La73b} are extensions of finitely many modules of the type
 $\ov{BP}/I(m)$. Since some power of $v_r$ annihilates our module, all these $m$ are $>r$. But, for such $m$, $Tor^{\ov{BP}}_i(\ov{BP}/I(m),K(r))=0$, for all $i$, as can be seen from the Koszul resolution
 of the $\ov{BP}$-module $\ov{BP}/I(m)$. Since $Tor$s commute with filtered colimits, we obtain that $Tor^{\ov{BP}}_i(\ov{BP}_{(c)}(M\otimes M^{\vee}),K(r))=0$, for all $i$. This implies that (for the codimension of support filtration $F_l$ on $\ov{BP}$) the map
 $$
 F_{c+1}\ov{BP}(M\otimes M^{\vee})\otimes_{\ov{BP}}K(r)\hookrightarrow
 F_c\ov{BP}(M\otimes M^{\vee})\otimes_{\ov{BP}}K(r)
 $$ 
 is injective (actually, an isomorphism).
 
 2) The space of $\ov{BP}$-generators of $\ov{BP}_{(c)}(M\otimes M^{\vee})$ can be identified with $\CH^c(M\otimes M^{\vee})/p$.  From \cite[proof of Theorem 4.1]{ACMLR} we know that for any subspace of generators, the relations in the respective $\ov{BP}$-submodule are generated by elements of positive co-dimension. Let $y$ be such a generator that $v_r\not\in\sqrt{Ann(y)}$ (we know that such $y$ exists). Let $N$ be the minimal quotient-module of $\ov{BP}_{(c)}(M\otimes M^{\vee})$, obtained by moding out a $\ov{BP}$-submodule generated by some subspace of generators with the property that the image
 of $y$ there is not annihilated by any power of $v_r$. I claim 
 that $N$ is finitely generated $\ov{BP}$-module. 
 Note that $N$ still has the property that the submodule generated by any subspace of generators of it has relations generated in positive co-dimension.
 Hence, if the submodule $\ov{BP}\cdot x$ generated by some generator $x$ 
 does not contain elements of the type $y\cdot v_r^k$ of positive co-dimension, it doesn't contain such elements at all (of any co-dimension). But since there are only finitely many elements of $\ov{BP}$ of bounded dimension, the space of generators of $N$ has a subgroup $C$ of finite index, so that for every
 $x\in C$, the submodule $\ov{BP}\cdot x$ doesn't
 contain elements of the form $y\cdot v_r^k$ of positive co-dimension. If $C$ is non-zero, this contradicts the minimality of $N$ (as we may mod-out the $\ov{BP}$-submodule, generated by $x$). Hence, $N$ has only finitely many generators. Since all the relations are in positive co-dimension, it is finitely presented (there are only finitely many such relations)
 
 Since $N$ is a finitely presented $BP_*BP$-module, by the results of Landweber \cite[Lemma 3.3]{La73b}, $N$ is an extension of finitely many
 $BP_*BP$-modules of the type $\ov{BP}/I(m)$. Since no power of $v_r$ annihilates the image of $y$ and so, $N$ itself, we see that among these pieces there should be, at least, one with $m\leq r$.
 Now, from the Koszul resolution of a $\ov{BP}$-module $\ov{BP}/I(m)$, we see that $\left(\ov{BP}/I(m)\right)\otimes_{\ov{BP}}K(r)=K(r)$, for $r\geq m$. Looking at the ``most outer'' piece with
 $m\leq r$ (and recalling that, for $l>r$, $Tor^{\ov{BP}}_i(\ov{BP}/I(l),K(r))=0$, for all $i$), we obtain that $N\otimes_{\ov{BP}}K(r)\neq 0$. Since $N$ is a quotient-module of $\ov{BP}_{(c)}(M\otimes M^{\vee})$, we get that
 $\ov{BP}_{(c)}(M\otimes M^{\vee})\otimes_{\ov{BP}}K(r)$
 is non-zero too. And so is $F_c\ov{BP}(M\otimes M^{\vee})\otimes_{\ov{BP}}K(r)$.
 
 Since we know that there exists $c$, such that $v_r\not\in\sqrt{Ann(\ov{BP}_{(c)}(M\otimes M^{\vee}))}$, combining cases 1) and 2) we obtain that $\ov{BP}(M\otimes M^{\vee})\otimes_{\ov{BP}}K(r)=K(r)(M\otimes M^{\vee})$ is non-zero. 
 Hence, $M^{K(r)}\neq 0$.
 \Qed
\end{proof}

Thus, our notion of the $p$-{\it level} agrees with that of the {\it type} in topology - see \cite[Definition 1.5.3]{Rav}.

\begin{rem}
 Since $K(1)$ is just a re-orientation (as in Quillen \cite{Qu71}, or Panin-Smirnov \cite{PS}) of the $p$-localised K-theory $K_0$, we see that ($p$-localised) Chow motives $M$ of $p$-level $\leq 1$
 are distinguished by the property that their K-motives are non-trivial. For motives of $p$-level $>1$, their $K$-motive is zero. So, these are 'K-phantom motives'.
 \Red
\end{rem}

Moreover, the Morava K-theory of $M$ disappers together with the ``derived versions'' of it.

\begin{prop}
 Let $M$ be a ($p$-localized) Chow motive of $p$-level $n$. Then the following conditions are equivalent:
 \begin{itemize}
  \item[$(1)$] $r<n$;
  \item[$(2)$] $Tor^{BP}_i(BP(M\otimes M^{\vee}),K(r))=0$, for all $i$.
 \end{itemize}
\end{prop}

\begin{proof}
 By the very definition, $(1)$ implies that some power of $v_r$ annihilates $BP(M\otimes M^{\vee})$. But by the results of Sechin \cite[Proposition 2.21]{S18b} and
 Landweber \cite[Lemma 3.3]{La73b}, this module is a union of finitely-presented $BP$-modules, each of which is an extension of finitely many
 modules of the type $BP/I(m)$. All these $m$s should be $>r$ due to mentioned annihilation. But as we discussed above, for $m>r$, $Tor^{BP}_i(BP/I(m),K(r))=0$, for all $i$. So, the same is true about $BP(M\otimes M^{\vee})$ which implies $(2)$.
 
 Conversely, since $Tor^{BP}_0(BP(M\otimes M^{\vee}),K(r))=K(r)(M\otimes M^{\vee})$, the $i=0$ case of $(2)$ implies $(1)$ by 
 Proposition \ref{In-Morava}.
 \Qed
\end{proof}

The following result shows that the $p$-level is stable under $\otimes$-powers.

\begin{lem}
 \label{tensor-pow} 
 Let $A^*$ be an oriented cohomology theory and $M\in Ob(Chow^A(k))$. Then $M\neq 0$ $\Rightarrow$
 $M^{\otimes n}\otimes (M^*)^{\otimes m}\neq 0$,
 $\forall n,m$, where $M^*=\Homi(M,T)$.
\end{lem}

\begin{proof}
 It is sufficient to prove it for $n=m=1$. This is a standard
 argument for tenzor rigid categories:
 $
 M\neq 0\Rightarrow id_M\neq 0\Rightarrow T\stackrel{\Delta}{\row} M\otimes M^*\,\,\text{is nonzero}\,\,
 \Rightarrow M\otimes M^*\neq 0
 $. 
 \Qed
\end{proof}

Applying this to Morava-motives of $M$, we obtain the first statement of the following:

\begin{prop}
 \label{p-level-tensor}
 The $p$-level of $M^{\otimes n}$ coincides with that of $M$.
More precisely, we have:
 The motives $M^{\otimes n}\otimes (M^{\vee})^{\otimes m}$, for all $(n,m)\neq (0,0)$, have the same $J$ and $\ov{J}$.
\end{prop}

\begin{proof}
 It is sufficient to compare $(1,0)$ and $(1,1)$. 
 In the notations of the proof of Lemma \ref{tensor-pow},\\
 $z\cdot id_M=0\Leftrightarrow z\cdot\Delta=0\Leftarrow z\cdot id_{M\otimes M^{*}}=0\Leftarrow z\cdot id_M=0$ in both
 $\Omega^*$ and $Gr\Omega^*$, where $M=(X,\rho)$ and $M^{\vee}=M^*(d)[2d]$ for $d=\ddim(X)$ is given by the dual projector.
 \Qed
\end{proof}

The following result allows to obtain a bound on the possible size of a torsion motive in terms of the $p$-level of it.

\begin{prop}
 \label{chow-level-r}
 Let $Y$ be a smooth variety, $z\in\CH^r_{\zz_{(p)}}(Y)$ and $\alpha\equiv v_{m}^k (mod\, I(m))\in BP$ be such that $z\cdot\alpha=0\in Gr BP(Y)$. Then
 \begin{itemize}
  \item[$(1)$] Either $z=0$, or $\displaystyle r\geq\frac{p^{m+1}-1}{p-1}$;
  \item[$(2)$] If $\displaystyle r\leq\frac{(2p-1)(p^m-1)}{p-1}$, then there exists an $\alpha$ as above with $k=1$.
 \end{itemize}
\end{prop}

\begin{proof}
The proof is based on {\it symmetric operations} of 
\cite{SOpSt}. The total symmetric operation
$\Phi:\Omega_{\zz_{(p)}}^*\row\Omega_{\zz_{(p)}}^*[t^{-1}]$
is the non-positive part of the total Quillen's type Steenrod operation on $\Omega^*$ divided by 'formal' $[p]$ - see \cite[Theorem 7.1]{SOpSt}. This operation can be extended to the graded algebraic cobordism $Gr\Omega^*$ and
to the $BP^*$-theory (as well as to the graded version of it) - see \cite{ACMLR}. 
By \cite[Proposition 7.14]{SOpSt}, the action of $\Phi$ on $Gr BP^*$ is described as
follows: for $z\in\CH^r(X)$, with $r>0$, and $u\in BP=\zz_{(p)}[v_1,v_2,\ldots]$, we have:
\begin{equation}
 \label{fizu}
\Phi(z\cdot u)=z\cdot t^{r(p-1)}\cdot{\mathbf i}^r\cdot\Phi(u)_{\leq -r(p-1)}, 
\end{equation}
where $\mathbf i$ is some integer, invertible in $\zz_{(p)}$. Here $\Phi_{\leq a}$ is the part of $\Phi$ of $t$-degree $\leq a$.

The idea is to 'divide' the relation $\alpha\cdot z=0$ by $v_{m}$ using symmetric operations (cf. \cite[Corollary 7.11]{SOpSt}). This is possible as long as codimension of $z$
is not very large.
Namely, by \cite[Proposition 7.16]{SS} we know about the action of $\Phi$ on $BP$ that
if $\alpha\equiv v_n^l\,(mod\,I(n))$, for $l>0$, then 
$\Phi_{-l(p-1)(p^n-1)-(p^n-1)}(\alpha)\equiv -v_n^{l-1}\,(mod\,I(n))$. If $\displaystyle r=codim(z)\leq\frac{p(p^{m}-1)}{p-1}$, then 
$-l(p-1)(p^{m}-1)-(p^{m}-1)\leq -r(p-1)$, and so,
(\ref{fizu}) applies. This shows by decreasing induction
on $l$, that for every $k\geq l\geq 0$ there is $\alpha_l\equiv v_{m}^l\,(mod\,I(m))$ such that
$\alpha_l\cdot z=0\in GrBP^*(Y)$. More precisely, we apply $\displaystyle\Phi_{r(p-1)-l(p-1)(p^{m}-1)-(p^{m}-1)}$ to $\alpha_l\cdot z=0$ and get
$\displaystyle\alpha_{l-1}=-{\mathbf i}^{-r}\cdot\Phi_{-l(p-1)(p^{m}-1)-(p^{m}-1)}(\alpha_l)$ of the needed form. Finally, we obtain $\alpha_0\equiv 1\,(mod\, I(m))$ such that $\alpha_0\cdot z=0\in Gr BP^*$. Hence, either $z=0$, or
$\displaystyle r>\frac{p(p^{m}-1)}{p-1}=\frac{p^{m+1}-1}{p-1}-1$. This proves $(1)$.

If $\displaystyle r\leq\frac{(2p-1)(p^m-1)}{p-1}$, then 
$-l(p-1)(p^{m}-1)-(p^{m}-1)\leq -r(p-1)$ as long as $l\geq 2$,
and so, there is $\alpha_1\equiv v_m(mod\, I(m))$ such that
$\alpha_1\cdot z=0\in GrBP^*(Y)$, which gives $(2)$.
 \Qed
\end{proof}

\begin{thm}
 \label{level-size}
 Let $M\!=\!(X,\rho)\!\in\! Ob(Chow(k,\zz_{(p)}))$ be a motive of $p$-level $n$. Then $\displaystyle\ddim(X)\geq\frac{p^{n}-1}{p-1}$.
\end{thm}

\begin{proof}
This follows from Proposition \ref{chow-level-r} (1) applied to
$Y=X^{\times 2}$, $r=\ddim(X)$, $m=n-1$ and $z=\rho$, $\alpha=v_{n-1}^k\in\ov{J}(M)$ .
 \Qed
\end{proof}

More generally, we obtain:

\begin{prop}
 \label{CH-small}
 Let $M=(X,\rho)$ be a torsion motive of $p$-level $n$.
 Then $\CH^i_{\zz_{(p)}}(M_F)=0$, for $\displaystyle i<\frac{p^{n}-1}{p-1}$ and any field extension $F/k$.
\end{prop}

\begin{proof}
 Observe that every element $\alpha\in BP$ annihilating $\rho$ will annihilate $\CH^*_{\zz_{(p)}}(M_F)$ in $GrBP$. It remains to apply
 Proposition \ref{chow-level-r} (1). 
 \Qed
\end{proof}

\begin{prop}
 \label{JM-small-dim}
 Let $M\!=\!(X,\rho)\in Ob(Chow(k;\zz_{(p)}))$ be a torsion motive of $p$-level $n$ with $\displaystyle\ddim(X)\leq\frac{(2p-1)(p^{n-1}-1)}{p-1}$. Then $\ov{J}(M)=I(n)=
 \sqrt{J(M)}$.
\end{prop}

\begin{proof} 
By Proposition \ref{chow-level-r} (2) applied to $z=\rho$, we
know that there exists some $\alpha\equiv v_{n-1}(mod\, I(n-1))\in \ov{J}(M)$.
Then it follows from Observation \ref{ovJ-LN} that 
$I(n)\subset\ov{J}(M)$. 
Indeed, if $\alpha\equiv v_{n-1}(mod\, I(n-1))$, then
$\alpha\equiv v_{n-1}(mod\, I^2)$, where $I=I(\infty)$ is the ideal generated by all $v_i$s. Applying appropriate Landweber-Novikov operations to $\alpha$, we
get elements $\alpha_i\equiv v_i(mod\, I^2)$ in $\ov{J}(M)$, for any $0\leq i\leq n-1$ - see, for example, \cite[Proposition 3.1]{ACMLR}. Then an increasing induction on $i$ shows that $v_i\in\ov{J}(M)$, for any $0\leq i\leq n-1$.
Finally, in light of Proposition
\ref{JM-JMt}, our embedding is an equality.
 \Qed
\end{proof}

In addition, we can observe that torsion motives are
not present in the motives of curves.

\begin{prop}
 \label{no-in-curves}
 If $M=(X,\rho)$ is a torsion motive, then $\ddim(X)>1$.
\end{prop}

\begin{proof}
 The fact that torsion motives can't be direct summands of 0-dimensional varieties is obvious, since the Chow-endomorphism rings of the latter are torsion-free.
 
 Suppose $X$ is a curve. To prove that $\rho=0$ it is sufficient to show that it acts trivially on Chow groups
 of $X$ over any field extension $E/k$. The action on $\CH^0(X)=\zz$ is clearly zero. By the same reason, the action on $\CH_0(X)$ passes through 0-cycles of degree $0$ and so, defines an action on $J(E)$ - the $E$-rational
 points of the Jacobian of $X$. It is given by an algebraic map $J\row J$ which should be the projection to
 the zero point, as it is torsion. Hence, $\rho=0$.
 \Qed
\end{proof}

Torsion motives do exist.

\begin{exa}
 \label{GBB}
 In \cite{BGS}, \cite{GS} and \cite{AO} examples of the, so-called, 'quasiphantom' subcategories
 of $D^b(coh(S))$ were constructed for some surfaces (over the field $\cc$). Namely, for Godeaux, Beauville and Burniat surfaces.
 It was shown by Gorchiskiy-Orlov in \cite[Proposition 2.2]{GO} that Chow motives of respective surfaces contain direct summands $M$ such that $\CH^*(M)=\CH^1(M)=\op{Pic}(S)_{tors}$. The latter group is $\zz/5\zz$, $(\zz/5\zz)^2$
 and $(\zz/2\zz)^6$, for Godeaux, Beauville, respectively, Burniat surface. Moreover, it follows from \cite[Proposition 2.3]{GO} that any $n\in\nn$ which annihilates $\op{Pic}(S)_{tors}$, annihilates $M$ itself.
 Thus, the motives $M$ are torsion motives. These were used by Gorchinsky-Orlov to construct the first known example of a 'phantom category' - see \cite{GO}. 
 
 Since our motives $M$ are killed by $5$, respectively, $2$, their ideals $J(M)$ involve single primes only. Moreover, since their K-motives (=$K(1)$-motives) are non-trivial (as $K_0(M)=\op{Pic}(S)_{tors}\neq 0$), it follows
 from Proposition \ref{In-Morava} that $p$-levels of $M$ are $1$. The same can be seen from Theorem \ref{level-size}, as
 motives of surfaces can't contain torsion motives of level
 $>1$. Thus, $\sqrt{J(M)}$ is $I(5,1)$, $I(5,1)$,
 $I(2,1)$ for the Godeaux, Beauville, respectively Burniat torsion motive. 
 
\end{exa}

\section{Motivic cohomology of torsion motives}
\label{section-three}

More subtle results on torsion motives can be obtained with
the help of their motivic cohomology

\subsection{Pre-morphisms of theories}
\label{subsection-three-one}

The results of this subsection hold over any field (of any characteristic).

\begin{defi}
 Let $A^*$ and $B^*$ be oriented cohomology theories in the sense of Levine-Morel - \cite[Definition 1.1.2]{LM} (no (LOC) axiom). A 'pre-morphism' of theories $G:A^*\row B^*$
 is an additive morphism of functors on $\smk$ which, in addition, respects maps of multiplication by the $1$-st Chern classes of line bundles.
\end{defi}

\begin{prop}
\label{laz-lin}
 Let $k$ be any field. A 'pre-morphism' of theories is $\laz$-linear.
\end{prop}

\begin{proof}
 Consider the map 
 $$
 X\times\left(\pp^{\infty}\times\pp^{\infty}\stackrel{Segre}{\row}\pp^{\infty}\right).
 $$
 Let $x^{A,B}, y^{A,B}, t^{A,B}$ be the 1-st Chern classes
 of the line bundles $O(1)$ lifted from various components,
 in $A^*$, respectively, $B^*$-theory. Then
 $Segre^*(t^C)=F_C(x^C,y^C)$, where $F_C$ is the formal group law of the oriented theory $C$.
 Let $\alpha\in A^*(X)$. Then, using the fact that $G$ respects multiplication by the 1-st Chern classes, we get:
 \begin{equation*}
 \begin{split}
 &\sum_{i,j}G(\alpha\cdot a_{i,j}^A)\cdot (x^B)^i(y^B)^j=
 \sum_{i,j}G(\alpha\cdot a_{i,j}^A(x^A)^i(y^A)^j)=
 G(\alpha\cdot F_A(x^A,y^A))=\\
 &G(\alpha\cdot Segre^*(t^A))=
 (id_X\times Segre)^*G(\alpha\cdot t^A)=
 (id_X\times Segre)^*(G(\alpha)\cdot t^B)=\\
 &G(\alpha)\cdot Segre^*(t^B)=G(\alpha)\cdot F_B(x^B,y^B)=
 \sum_{i,j}G(\alpha)\cdot a_{i,j}^B\cdot (x^B)^i(y^B)^j.
 \end{split}
 \end{equation*}
 Comparing coefficients at $(x^B)^i(y^B)^j$, we obtain:
 $G(\alpha\cdot a_{i,j}^A)=G(\alpha)\cdot a_{i,j}^B$.
 Since, the coefficents of the universal formal group law
 generate the Lazard ring $\laz$ additively, we get that
 $G$ is $\laz$-linear.
 \Qed
\end{proof}

\begin{prop}
\label{Chern-push}
 Let $k$ be an arbitrary field and $G:A^*\row B^*$ be an additive operation, where $A^*$ and $B^*$ are oriented cohomology theory in the sense of Levine-Morel and $B^*$, in addition, satisfies the weak form of (LOC) axiom (exactness in the middle) as in Panin \cite[Definition 1.1.7]{P-RR}. 
 Then the following conditions are equivalent:
 \begin{itemize}
  \item[$(1)$] $G$ respects push-forwards.
  \item[$(2)$] $G$ commutes with maps of multiplication by the $1$-st Chern classes of line bundles.
 \end{itemize} 
\end{prop}

\begin{proof}
$(1)\row (2)$: Let $L$ be a line bundle on $X$. We will denote the total space of it by the same symbol $L$. The zero section
provides a regular closed embedding $f:X\row L$, s.t.
$f_*f^*$ coincides with the multiplication by the 1-st Chern class of $L$. Since $G$ commutes with $f^*$ and $f_*$, it commutes with the multiplication by this Chern class.

$(2)\row (1)$: Since any projective morphism can be decomposed into a composition of a regular embedding and a
trivial projective fibration: 
$X\row\pp^n\times Y\row Y$, it is sufficient to check the statement for these two types of maps.

1) Let $f:X\row Y$ be a regular emebedding with the normal
bundle $N_f$. Let $\mu_i^{B}$, $i\in\ov{n}=\{1,\ldots,n\}$ be the $B^*$-{\it roots} of $N_f$.
Then by the general Riemann-Roch Theorem - \cite[Theorem 5.19]{PO}, for any $\alpha\in A^*(X)$, 
\begin{equation*}
 G(f_*(\alpha))=f_*\operatornamewithlimits{Res}_{t=0}
 \frac{G(\prod_{i\in\ov{n}}x_i^A\cdot\alpha)
 (x_i^B=t+_B\mu_i^B|_{i\in\ov{n}})\cdot\omega^B_t}
 {t\cdot\prod_{i\in\ov{n}}(t+_B\mu_i^B)},
\end{equation*}
where $x_i^{A,B}$ are the 1-st Chern classes of the line bundle $O(1)$ lifted from the $i$-th component in
$A^*(X\times(\pp^{\infty})^{\times n})$, respectively,
$B^*(X\times(\pp^{\infty})^{\times n})$, and we make a substitution instead of $x_i^B$-variables.
Since $G$ commutes with the multiplication by the 1-st Chern classes, this can be rewritten as 
\begin{equation*}
 f_*\operatornamewithlimits{Res}_{t=0}
 \frac{G(\alpha)\cdot\prod_{i\in\ov{n}}(t+_B\mu_i^B)
 \cdot\omega^B_t}{t\cdot\prod_{i\in\ov{n}}(t+_B\mu_i^B)}=
 f_*(G(\alpha)).
\end{equation*}

2) Let $\pi:Y\times\pp^n\row Y$ be the projection and 
$\alpha\in A^*(Y\times\pp^n)$. By the Projective Bundle Theorem, it can be uniquely written as
$\alpha=\sum_{i=0}^n\pi^*(a_i)\cdot(\xi^A)^i$, where $\xi^A=c_1^A(O(1))$ and $a_i\in A^*(Y)$. Then, by the projection formula,
$\pi_*(\alpha)=\sum_{i=0}^na_i\cdot [\pp^{n-i}]^A$.
So, by Proposition \ref{laz-lin} and (2), 
\begin{equation*}
\begin{split}
&G(\pi_*(\alpha))=\sum_{i=0}^nG(a_i\cdot [\pp^{n-i}]^A)=
\sum_{i=0}^nG(a_i)\cdot [\pp^{n-i}]^B=\\
&\pi_*(\sum_{i=0}^n\pi^*G(a_i)\cdot(\xi^B)^i)=
\pi_*(G(\alpha)).
\end{split}
\end{equation*}
 \Qed
\end{proof}

\begin{rem}
 Proposition \ref{Chern-push} shows that, if the target theory $B^*$ has (a weak form of) localisation, then 'pre-morphisms' of theories can be defined as additive transformations respecting both pull-backs and push-forwards (for projective morphisms). This agrees with and provides a simplification of the \cite[Definition 2.10]{PO}. 
 So, the only difference with genuine 'morphisms' of theories is that 'pre-morphisms' don't have to respect the
 multiplicative structure.
 \Red
\end{rem}

\subsection{Milnor's operations}

In this Subsection our ground field is any field of characteristic zero.

Let 
$$
Q_r:\hm^{a,b}(X,\ff_p)\row\hm^{a+2p^r-1,b+p^r-1}(X,\ff_p)
$$ 
be the motivic analogues of Milnor's operations
introduced by Voevodsky in \cite{VoOP}.
These operations are differentials: $Q_r\circ Q_r=0$ and
anti-commute with each other: $Q_r\circ Q_l=-Q_l\circ Q_r$,
for $r\neq l$. For $p=2$ and $I\subset\ov{\nn}=\nn\cup\{0\}$, let us denote as $Q_I$ the composition $\circ_{i\in I} Q_i$. Then the behavior of Milnor's operations with respect to the multiplicative structure is given by: $Q_r(x\cdot y)=\mu(\Delta(Q_r)(x\otimes y))$, where
\begin{equation}
\label{coprod}
  \Delta(Q_r)=
  \begin{cases}
   Q_r\otimes 1+1\otimes Q_r,\hspace{2mm}\text{for}\,\,p>2;\\
   \sum_{2^I+2^J=2^r}Q_I\otimes Q_J\cdot\{-1\}^{|I|+|J|-1},
   \hspace{2mm}\text{for}\,\,p=2,
  \end{cases}
\end{equation}
where $2^I=\sum_{i\in I}2^i$.

\begin{prop}
 \label{milnor-chow-linear}
 Milnor's operations are Chow group linear.
\end{prop}

\begin{proof}
 Milnor's operations act trivially on Chow groups of smooth varieties, as
 their bi-degree is $(p^r-1)[2p^r-1]$ and so, the potential target resides in bi-degrees, where smooth varietieties have no motivic cohomology (above the slope $=2$ line). 
 Then the coproduct formula (\ref{coprod}) shows that
 $Q_r(u\cdot\alpha)=u\cdot Q_r(\alpha)$, for any Chow group
 element $u$. 
 \Qed
\end{proof}

As 1-st Chern classes of line bundles reside in Chow groups, we immediately get the following.

\begin{cor}
 \label{milnor-premorphism}
 Milnor's operations are pre-morphisms of theories.
\end{cor}

Taking into account that motivic cohomology satisfy the conditions of Proposition \ref{Chern-push},
combining further Proposition \ref{milnor-chow-linear} with
Proposition \ref{Chern-push} we obtain the following result (cf. \cite[Lemma 9.1]{YaABP}).

\begin{cor}
 \label{milnor-corresp}
 Milnor's operations commute with (push-forwards for) correspondences and so, define a morphism of functors on the category of Chow motives.
\end{cor}

\begin{proof}
 Let $\varphi:X\rightsquigarrow Y$ be a correspondence (between smooth projective varieties). Then $\varphi_*$ is the composition of a pull-back, followed by a multiplication by a certain Chow group element, followed by a push-forward. By Propositions \ref{milnor-chow-linear} and
 \ref{Chern-push}, Milnor's operations commute with all these ingredients.
 \Qed
\end{proof}

\begin{rem}
 Of course, nothing of this sort is true for 'non-pure' motives, as Milnor's operations are not linear with respect to (non-pure) motivic cohomology elements.
 \Red
\end{rem}

Following Voevodsky, let us introduce the spectrum
${\mathbf\Phi}_r$ in $\shk$ from the distinguished triangle
$$
T^{p^r-1}\wedge{\mathbf\op{H}_{\ff_p}}\stackrel{u}{\lrow}{\mathbf\Phi}_r\stackrel{v}{\lrow}{\mathbf\op{H}_{\ff_p}}\stackrel{Q_r}{\lrow}
S^1_s\wedge T^{p^r-1}\wedge{\mathbf\op{H}_{\ff_p}},
$$
where ${\mathbf\op{H}_{\ff_p}}$ is the {\it Eilenberg-MacLane spectrum}.
It follows from \cite[Lemma 3.23]{VVMilnorConj} that
${\mathbf\Phi}_r$ has a structure of an $\op{MGL}$-module and so, an orientation compatible with this distinguished triangle. In other words, that the respective cohomology theory has a structure of push-forwards compatible with push-forwards for motivic cohomology with $\ff_p$-coefficients. Thus, we also have the action of (algebraic cobordism of Levine-Morel) $\Omega^*$-correspondences on our distinguished triangle.

\begin{prop}
 \label{corresp-Margolis}
 Let $X$ and $Y$ be smooth projective varieties and $\ffi^{BP}:X\rightsquigarrow Y$ be a $BP^*$-correspondence. Let $z\in BP$ be such $\nu_r$-element
 that $z\cdot\ffi^{BP}=0\in GrBP$. Then the Chow trace $\ffi^{CH}$ acts trivially on the $r$-th Margolis cohomology. That is, if $Q_r(x)=0$, for some $x\in\hm(X)$, then $\ffi^{CH}_*(x)=Q_r(y)$, for some $y\in\hm(Y)$.
\end{prop}

\begin{proof}
Combining the action of $\Omega^*$-correspondences with \cite[Proposition 3.24]{VVMilnorConj}, we get a diagram
$$
\xymatrix{
& {\mathbf\Phi}_r^{*+2m,*'+m}(Y) \ar[d]_{v} \ar[dr]^(0.45){\cdot z} & \\
{\mathbf\Phi}_r^{*,*'}(X) \ar[ur]^{\ffi^{BP}_*} \ar[dr]^(0.3){\cdot z} \ar[d]_{v}
& \hm^{*+2m,*'+m}(Y,\ff_p) \ar[r]_(0.4){u} &
{\mathbf\Phi}_r^{*+2m-2(p^r-1),*'+m-(p^r-1)}(Y)\\
\hm^{*,*'}(X,\ff_p) \ar[ur]^(0.6){\ffi^{CH}_*} 
|!{[u];[r]}\hole \ar[r]_(0.35){u} & {\mathbf\Phi}_r^{*-2(p^r-1),*'-(p^r-1)}(X) \ar[ru]^{\ffi^{BP}_*}
}
$$
commuting up to an invertible element of $\ff_p$ (here $m$ is the degree of $\ffi$). 
Since ${\mathbf\Phi}_r^{2*,*}(V)=BP^*(V)\otimes_{BP}\zz/p[v_r]/(v_r^2)$, we see that, for $x\in BP^*(V)$,
if $z\cdot x=0\in GrBP(V)$, then $z\cdot x=0\in{\mathbf\Phi}_r(V)$.
Hence, the composition $u\circ\ffi^{CH}_*\circ v$ is zero, which means exactly that $\ffi^{CH}_*$ acts as zero on
$\kker(Q_r)/\op{Im}(Q_r)$.
 \Qed
\end{proof}

\begin{cor}
 \label{torsion-M-Margolis}
 If $M$ is a torsion motive such that $\ov{J}(M_{\zz_{(p)}})\supset I(n)$, then the $r$-th Margolis cohomology of $M$ are trivial, for all $0\leq r<n$.
\end{cor}

\begin{proof}
Apply Proposition \ref{corresp-Margolis} to $\ffi=\rho_M$.
 \Qed
\end{proof}

\begin{rem}
 \label{ss-Kr-k}
 More generally, one can show that if $\ov{J}(M_{\zz_{(p)}})$ contains some $\alpha\equiv v_{n-1}^k(mod\, I(n-1))$, then the spectral sequence $\hm^{*,*'}(M,\zz/p)\Rightarrow\op{K}(r)(M)$
 (with the target being zero by Proposition \ref{In-Morava})
 should degenerate after $k$ differentials, for all $0\leq r<n$.
 \Red
\end{rem}

\subsection{Consequences for torsion motives}

In this Subsection, our ground field is any field of characteristic zero.
It appears that motivic cohomology of torsion motives are absent in some regions depending on the level.

\begin{prop}
 \label{mot-coh-abs}
 Let $M=(X,\rho)$ be a torsion motive such that $\ov{J}(M_{\zz_{(p)}})$
 contains $\alpha\equiv v_{n-1}(mod\,I(n-1))$. Then
 $\hm^{a,b}(M_F,\zz/p)=0$, for $a< n$, as well as for $b< n$, over any $F/k$.
\end{prop}

\begin{proof}
 By the Beilinson-Lichtenbaum ``conjecture'' it is sufficient to prove the statement for $b< n$. Induction on $b$. For $b<0$ we know it. 
 Suppose, $u\in\hm^{a,b}(M,\zz/p)$ is a non-zero element. 
 Again, by the Beilinson-Lichtenbaum and the inductive assumption we can assume that $a\geq b$ and, in the case $a=b$, $u$ is not in the image of $Q_0$. 
 Now by induction on $0\leq r< n$ we show that
 $Q_r\circ\ldots\circ Q_0(u)$ is non-zero. Suppose not.
 Then by Corollary \ref{torsion-M-Margolis},
 $Q_{r-1}\circ\ldots\circ Q_0(u)=Q_r(v)$, for some $v$.
 But the ``round'' degree of $v$ is equal to
 $\displaystyle b+(p^0-1)+\ldots+(p^{r-1}-1)-(p^r-1)=
 b+\frac{p^r-1}{p-1}-r-(p^r-1)<b$. Thus, $v=0$ by the inductive assumption - a contradiction. Hence, $Q_{n-1}\circ\ldots\circ Q_0(u)\neq 0$.
 But this element resides in some group $\hm^{c,d}(M,\zz/p)$
 with $c>2d$ - a contradiction. Therefore, $u=0$. 
 \Qed
\end{proof}

Combining it with the Beilinson-Lichtenbaum ``conjecture'', we obtain:

\begin{cor}
 \label{et-vanish-a}
 Let $M=(X,\rho)$ be a torsion motive such that $\ov{J}(M_{\zz_{(p)}})$
 contains $\alpha\equiv v_{n-1}(mod\,I(n-1))$. Then
 $\op{H}^a_{et}(M_F,\zz/p)=0$, for any $a<n$ and any 
 $F/k$.
\end{cor}

\begin{rem}
 \label{ab-k-less-p}
 In view of the Remark \ref{ss-Kr-k}, similar arguments imply 
 that the result still holds if $\ov{J}(M_{\zz_{(p)}})$
 contains some $\alpha\equiv v_{n-1}^k(mod\, I(n-1))$, for
 $k<p$.
 \Red
\end{rem}

In the case of an algebraically closed field, we obtain more.

\begin{prop}
 \label{sm-Met-kbar}
 Let $M\!=\!(X,\rho)$ be a torsion motive such that $\ov{J}(M_{\zz_{(p)}})$
 contains $\alpha\!\equiv\! v_{n-1}(mod\,I(n\!-\!1))$. Then
 $\op{H}^a_{et}(M_{\kbar},\zz/p)=0$, for any $a<n$ and 
 for $a> 2d-n$, where $d=\ddim(X)$. In other words, the singular cohomology (with $\zz/p$-coefficients)
 $H^a$ of the topological realization of $M$ are absent in this range.
\end{prop}

\begin{proof}
 Let $M^{\vee}=\Homi(M,\zz(d)[2d])$ be the dual direct summand. Then $\ov{J}(M^{\vee})=\ov{J}(M)$ and, by Corollary \ref{et-vanish-a}, $\op{H}^a_{et}(M_{\kbar},\zz/p)=0$, for $a<n$, and the same is true for $M^{\vee}$.
 Since $\op{H}^a_{et}(M_{\kbar},\zz/p)=\left(\op{H}^{2d-a}_{et}(M^{\vee}_{\kbar},\zz/p)\right)^{\!\!*}$, we get that
 $\op{H}^a_{et}(M_{\kbar},\zz/p)=0$, for $a> 2d-n$.
 \Qed
\end{proof}

The above Proposition shows that a torsion motive of a high $p$-level can't reside just ``under the surface'', but should be ``sufficiently deep'' inside the motive of the respective variety.\\

For a number $x\in\rr$, denote as $\lceil x\rceil$ the smallest integer $\geq x$. 

\begin{prop}
 \label{prav-dim-et-vanish}
 Let $M=(X,\rho)\in Ob(Chow(k;\zz_{(p)}))$ be a torsion motive of $p$-level $n$ with\\
 $\displaystyle\ddim(X)<\lceil\frac{n}{2}\rceil+\frac{p^{n}-1}{p-1}$.
 Then $M_{et}|_{\kbar}=0$.
\end{prop}

\begin{proof}
For $n=1$ we know it from Proposition \ref{no-in-curves}.
So, assume that $n> 1$.
Since $\displaystyle\lceil\frac{n}{2}\rceil+\frac{p^{n}-1}{p-1}-1\leq\frac{(2p-1)(p^{n-1}-1)}{p-1}$, for $n> 1$, from Proposition \ref{JM-small-dim} we obtain that $I(n)\subset\ov{J}(M)$. Then Proposition \ref{sm-Met-kbar} implies that
$\op{H}^a_{et}(M_{\kbar},\zz/p)$ can be non-zero only for
$n\leq a\leq 2d-n$, where $d=\ddim(X)$.

It follows from Corollary \ref{torsion-M-Margolis} and the Beilinson-Lichtenbaum ``conjecture'' that $Q_r$
is an exact differential on $\op{H}^*_{et}(M_{\kbar},\zz/p)$, for any
$0\leq r<n$. This means that $\op{H}^*_{et}(M_{\kbar},\zz/p)$ is a free module over $\Lambda_{\zz/p}(Q_0,\ldots,Q_{n-1})$. But the ``height'' of this algebra is the
``square'' degree of $Q_{n-1}\circ\ldots\circ Q_0$, which is
$\displaystyle(2p^0-1)+\ldots+(2p^{n-1}-1)=2\frac{p^{n}-1}{p-1}-n>2d-2n$. Hence, $\op{H}^*_{et}(M_{\kbar},\zz/p)=0$. Since the category of etale motives with $\zz/p$-coefficients over $\kbar$ is the derived category of $\ff_p$-vector spaces, we obtain: $M_{et}|_{\kbar}(mod\, p)=0$, and so,
$M_{et}|_{\kbar}=0$, as $M$ is killed by some power of $p$.
 \Qed
\end{proof}

Corollary \ref{torsion-M-Margolis} permits to structurize the motivic cohomology of $M$ (provided $\ov{J}(M_{\zz_{(p)}})$ contains a $\nu_{n-1}$ element). Let us denote as $\Lambda$ the ring
$\Lambda_{\ff_p}(Q_0,\ldots,Q_{n-1})$.

\begin{prop}
 \label{mot-coh-u-v}
 Suppose $I(n)\subset\ov{J}(M_{\zz_{(p)}})$. Then 
 $\hm^{*,*'}(M,\zz/p)$ has a filtration $F^{(a,b)}$
 by $\Lambda[\tau]$-modules, where
 $(a,b)$ runs through all pairs $2b\geq a\geq b\geq 0$ 
 ordered lexicographically, $F^{(a,b)}/F^{<(a,b)}$ is a 
 direct sum of modules of the type
 $\Lambda[\tau]\cdot u_{\alpha}$
 and 
 $\Lambda[\tau]/(\tau^{l_{\beta}})\cdot v_{\beta}$, where $\ddeg(u_{\alpha})=\ddeg(v_{\beta})=(b)[a]$, $0\leq l_{\beta}\leq a-b-1$, and
 the quotient module $G=\hm^{*,*'}(M,\zz/p)/F^{(a,b)}$ satisfies: $G^{c,d}=0$, for $c<a$ and for $c=a$, $d\leq b$.
\end{prop}

\begin{proof} We can take $F^{<(0.0)}=0$. Suppose, we already constructed $F^{<(a,b)}$. By inductive assumption, we have an
exact sequence:
$$
0\row F^{<(a,b)}\row\hm^{*,*'}(M,\zz/p)\stackrel{g}{\row} G
\row 0,
$$
where $G^{c,d}=0$, for $c<a$ and for $c=a$, $d<b$. Then $Q_r$, $0\leq r\leq n-1$ act as exact differentials on $G$ as well.
By the standard arguments, using the fact that $G$ is trivial below the level $a$, we
obtain that $Q_{n-1}\circ\ldots\circ Q_0$ is injective on 
$G^{a,*}$. 
We have a filtration $G^{a,b}_l=\kker(\cdot\tau^l)$ on $G^{a,b}$. Let $G^{a,b}_{\infty}=\cup_l G^{a,b}_l$.
By the Beilinson-Lichtenbaum ``conjecture'', $G^{a,b}_{\infty}=G^{a,b}_{a-b-1}$. Choose an $\ff_p$-vector
subspace $H_l$ of $G^{a,b}_l$ that $G^{a,b}_l=G^{a,b}_{l-1}\oplus H_l$ and a subspace $H_{\infty}$ of $G^{a,b}$ such 
that $G^{a,b}=G^{a,b}_{\infty}\oplus H_{\infty}$. Then
$$
H=\left(\operatornamewithlimits{\oplus}_{0\leq l\leq a-b-1}\Lambda[\tau]/(\tau^l)\otimes_{\ff_p}H_l\right)\oplus\left(\Lambda[\tau]\otimes_{\ff_p}H_{\infty}\right)
$$
is a $\Lambda$-submodule of $G$ such that $H^{a,b}=G^{a,b}$.
It remains to take $F^{(a,b)}=g^{-1}(H)$. 
Induction step is proven. 
 \Qed
\end{proof}

This permits to improve a bit the bound of Theorem \ref{level-size}, at least, for algebraically closed fields.

\begin{prop}
 \label{plyus-odin}
 Let $M=(X,\rho)\in Ob(Chow(k;\zz_{(p)}))$ be a torsion motive of $p$-level $n$ with\\
 $\displaystyle\ddim(X)=\frac{p^{n}-1}{p-1}+i$, where
 $\displaystyle i<\lceil\frac{n}{2}\rceil$ and $i\leq 1$. Then $M_{\kbar}=0$.
\end{prop}

\begin{proof}
 By Proposition \ref{JM-small-dim}, $\ov{J}(M)=I(n)$. Hence, Proposition \ref{mot-coh-u-v} applies and 
 $\hm^{*,*'}(M_F,\zz/p)$ is an extension of modules of the type
 $\Lambda[\tau]\cdot u_{\alpha}$ and 
 $\Lambda[\tau]/(\tau^{l_{\beta}})\cdot v_{\beta}$, where
 $\ddeg(v_{\beta})=(b)[a]$ and $l_{\beta}\leq a-b-1$.
 Since $\displaystyle\ddim(X)\leq\frac{p^{n}-1}{p-1}+1$
 and $Q_r$ has the ``diagonal degree'' $p^r$, the generators
 $v_{\beta}$ (as well as $u_{\alpha}$) of our modules reside on the diagonals with numbers $\leq 1$. But multiplication by $\tau$ is injective on such diagonals by
 the Beilinson-Lichtenbaum ``conjecture''. Hence, the modules
 of the second kind ($\tau$-torsion ones) are absent, and so,
 multiplication by $\tau$ is injective on $\hm^{*,*'}(M_F,\zz/p)$, for any extension $F/k$. Consider now some finitely 
 generated extension $F/\kbar$. From Proposition 
 \ref{prav-dim-et-vanish}, $M_{et}|_{\kbar}=0$, which implies
 that $M_{et}|_F=0$ $\Rightarrow$ $\op{H}^*_{et}(M_F,\zz/p)=0$.
 Since multiplication by $\tau$ is injective, by the Beilinson-Lichtenbaum ``conjecture'' we obtain that
 $\hm^{*,*'}(M_F,\zz/p)=0$ for all finitely generated extensions $F/\kbar$. As $M$ is killed by some power of $p$,
 the same is true about integral motivic cohomology. In particular, all Chow groups of $M$ are zero, for any such $F$. Then $M_{\kbar}=0$.
 \Qed
\end{proof}

\begin{cor}
 \label{sharper-m-tri}
 Let $k=\kbar$ and $M=(X,\rho)$ be a torsion motive of 
 $p$-level $n\leq 4$. Then 
 $$
 \displaystyle\ddim(X)\geq 
 \lceil\frac{n}{2}\rceil+\frac{p^{n}-1}{p-1}.
 $$
\end{cor}

In the light of expected properties of Chow motives (in
particular, {\it Rost nilpotence conjecture}), it is natural
to propose:

\begin{conj}
 \label{gip-razmer} Let $M=(X,\rho)$ be a
  tosion motive of $p$-level $n$. Then 
  $\displaystyle\ddim(X)\geq 
 \lceil\frac{n}{2}\rceil+\frac{p^{n}-1}{p-1}$.
\end{conj}

\section{Torsion motives of surfaces}
\label{section-four}

\subsection{Godeaux torsion motive}

The aim of this subsection is to have a closer look at the torsion direct summand in the motive of the Godeaux surface.
Our ground field will be the field of complex numbers $\cc$. 

Let $X$ be a {\it Godeaux surface}, that is, the quotient of the Fermat quintic $x_1^5+x_2^5+x_3^5+x_4^5=0$ in $\pp^3$ by the $\zz/5$-action given by $\sigma(x_1,x_2,x_3,x_4)=(\xi x_1,\xi^2 x_2,\xi^3 x_3,\xi^4 x_4)$, for a generator $\sigma$ of $\zz/5$, where $\xi$ is a primitive root of $1$ of degree $5$.

It is a smooth projective surface whose singular cohomology
$\HH^*(X,\zz)=\HH^*_{sing}(X(\cc),\zz)$ are given by:
$\HH^i(X,\zz)=\zz$, for $i=0,4$; $\HH^1(X,\zz)=0$, 
$\HH^3(X,\zz)=\zz/5$; $\HH^2(X,\zz)=
\zz^9\oplus\zz/5$ - \cite[Section 2]{BGS}.

It was shown by Gorchinskiy-Orlov \cite[Proposition 2.2]{GO} that the torsion part of the cohomology corresponds 
to a direct summand of $M(X)$. It is not difficult to
see that the respective projector can be made self-dual.

\begin{prop}
 \label{M-dual}
 There exists a self dual with respect to $\Homi(-,\zz(2)[4])$ direct summand $M$ of $M(X)$ such that
 $\HH^*(M,\zz)=\HH^*(X,\zz)_{tors}$.
\end{prop}

\begin{proof}   
By \cite[Corollary 4.15]{BGS} there is a $\zz$-basis
of the free part of the Picard group $N(X)=\zz^9$ that
the respective intersection matrix $A$ is an invertible
(integral valued symmetric) matrix. Let $e_i,\,i=1,...,9$ be such a basis, and let $B$ be some
symmetric integral valued $9\times 9$ matrix. We can assign to
it the cycle $\Phi_{B}=\sum_{i,j}b_{i,j}\cdot e_i\times e_j\in\CH^2(X\times X)$. If $f=\sum_k\lambda_ke_k\in N(X)$, then $(\Phi_{B})_*(f)=B\cdot A\cdot f$. Thus, if
we take $B=A^{-1}$, then $(\Phi_{B})_*$ will act identically on $N(X)$. By the same reason, 
$\rho_1=\Phi_{A^{-1}}$
is a projector. Since $A^{-1}$ is a symmetric matrix, 
this projector is symmetric. It splits $(\zz(1)[2])^{\oplus 9}$ from the motive of $X$. In addition,
we have a projector $\rho_2=[X\times pt]+[pt\times X]$
orthogonal to $\rho_1$, and so, the symmetric projector $\rho=\rho_1+\rho_2$ splits the free part from the cohomology of $X$. That is,
$M(X)=\rho M(X)\oplus (1-\rho) M(X)$, where 
$\rho M(X)=\zz\oplus(\zz(1)[2])^{\oplus 9}\oplus
\zz(2)[4]$ and $\HH^*((1-\rho)M(X),\zz)=\HH^*(X,\zz)_{tors}$. Let us denote $(1-\rho)M(X)$ as $M$.
Since $(1-\rho)$
is symmetric, we have a natural identification
$\Homi(M,\zz(2)[4])=M$.
\phantom{a}\hspace{5mm}
\Qed
\end{proof}

Because $\HH^i(M,\zz)=\zz/5$, for
$i=2,3$ and zero for all other $i$, we obtain that
$\HH^1(M,\zz/5)={\mathbf u}\cdot\zz/5$, $\HH^2(M,\zz/5)=Q_0({\mathbf u})\cdot\zz/5\oplus {\mathbf v}\cdot\zz/5$, $\HH^3(M,\zz/5)=Q_0({\mathbf v})\cdot\zz/5$, where $Q_0$ is the Bockstein, with all other cohomology groups trivial. 

By the Artin comparison theorem 
\cite[Exp.XI, Theorem 4.4]{SGA4}, $\HH^*_{et}(X,\zz/5)=
\HH^*(X(\cc),\zz/5)$. Hence, $\HH^*_{et}(M,\zz/5)=\HH^*(M,\zz/5)$. 

\begin{prop}
 \label{Het-Xk}
 Suppose, $N$ is a Chow motive over $\cc$ and $k/\cc$ is
 some field extension and $l$ is prime. Then
 $$
 \HH^*_{et}(N_k,\zz/l)=\HH^*(N(\cc),\zz/l)\otimes_{\zz/l}
 \HH^*_{et}(k,\zz/l).
 $$
\end{prop}

\begin{proof}
 Since $\HH^*(N(\cc),\zz/l)=\HH^*_{et}(N_{\cc},\zz/l)$, we have a natural map
 $$
 \HH^*(N(\cc),\zz/l)\otimes_{\zz/l}\HH^*_{et}(k,\zz/l)\row
 \HH^*_{et}(N_k,\zz/l)
 $$ 
 which we claim to be an isomorphism.
 Our motive $N$ is a direct summand in $M(Y)$ for some
 smooth projective $\cc$-variety $Y$. It is sufficient to prove the statement for $M(Y)$.
 
 We have the Hochschild-Serre spectral sequence:
 $$
 E_2=\HH^*_{et}(k,\HH^*_{et}(Y_{\ov{k}},\zz/l))\Rightarrow
 \HH^*_{et}(Y_k,\zz/l).
 $$
 The $E_2$-term of this sequence is a free 
 $\HH^*_{et}(k,\zz/l)$-module with the basis 
 $\HH^0_{et}(k,\HH^*_{et}(Y_{\kbar},\zz/l))=
 \HH^*_{et}(Y_{\kbar},\zz/l)=\HH^*(Y(\cc),\zz/l)$, since
 the natural map $\HH^*_{et}(Y_{\cc},\zz/l)\row
 \HH^*_{et}(Y_{\kbar},\zz/l)$ is an isomorphism by the 
 smooth base change theorem \cite[Chapter VI, Corollary 4.3]{Miln}, and so, the 
 $Gal(\kbar/k)$-action on the latter module is trivial.
 Differentials in this sequence are $\HH^*_{et}(k,\zz/l)$-linear, while the basis elements survive in the $E_{\infty}$-term by the same smooth base change result, since these come from 
 $\HH^*_{et}(Y_{\cc},\zz/l)$. Hence, the sequence degenerates from the $E_2$-term and so, our original map is an isomorphism.
 \Qed
\end{proof}

We have the Brown-Gersten-Quillen type spectral sequence
$$
E_1^{p,q}(coniveau)=\oplus_{x\in X^{(p)}}\HH^q_{et}(k(x),\zz/l)\Rightarrow
\HH^{q+2p}_{et}(X_k,\zz/l),
$$
inducing the coniveau filtration on $\HH^*_{et}(X_k,\zz/l)$. Staring with the $E_2$-page this spectral sequence can be identified with the $\tau$-spectral sequence (see, for example, \cite[Section 2]{TY}).
$$
E_1^{a,b}(tau)=\moco{a}{b}(X_k,\km)\Rightarrow \HH^{a}_{et}(X_k,\zz/l)
$$
given by the exact pair
$$
\xymatrix{
& \oplus_{a,b}\moco{a}{b}(X_k,\km) \ar[dr]^{(-1)[1]} & \\
\oplus_{a,b}\moco{a}{b}(X_k,\zz/l) \ar[ur] & & 
\oplus_{a,b}\moco{a}{b}(X_k,\zz/l) 
\ar[ll]^{(1)}_{\cdot\tau}
}
$$
coming from the distinguished triangle
$\zz/l(-1)\stackrel{\tau}{\row}\zz/l\row\km\row\zz/l(-1)[1]$ (where $\tau$ corresponds to some choice of
the primitive $l$-th root of $1$ in $\cc\subset k$).
Note that $k_M$ belongs to the heart of the homotopic $t$-structure \cite{Deg} and is represented there by the Milnor's K-theory mod $p$ Rost cycle module $k^*_M=K^*_M/p$ (so, the name).
We have: $E_2^{p,q}(coniveau)=E_1^{2p+q,p+q}(tau)$, by the results of Deligne and Parajape - see \cite[Theorem 2.4]{TY}.
On the level of associated filtrations on etale cohomology, this is reflected by the fact
that elements of 
$\HH^n_{et}(X_k,\zz/l)=\moco{n}{n}(X_k,\zz/l)$ 
supported in codimension $\geq c$ are exactly those
which are divisible by $\tau^c$ in $\moco{*}{*'}$.

The differential $d_r$ of the $\tau$-sequence acts
as follows: $d_r:E_r^{a,b}\row E_r^{a+1,b-r}$ and,
for a $d$-dimensional variety, $E_1^{a,b}$ is non-zero
only for $a-2b\leq 0\leq a-b\leq d$. In particular, the sequence
degenerates at $E_d$. 
The diagonal part $E_1^{a,a}$ is nothing else, but
the {\it unramified cohomology}
$\HH^a_{nr}(k(X)/k,\zz/l)$ of the generic point of $X$.

If the ground field $k=\kbar$ is algebraically closed, then $\HH^q_{et}(\kbar(x),\zz/l)=0$,
for points of co-dimension $p>d-q$. Hence, the $E(tau)$
spectral sequence is concentrated, in this case, in the
region $a-2b\leq 0\leq a-b$; $b\leq d$. The $E_{\infty}$
page of it provides us with some sort of ``Hodge half-diamond'' describing the codimension of support of
elements in $\HH^*(X(\cc),\zz/l)$.

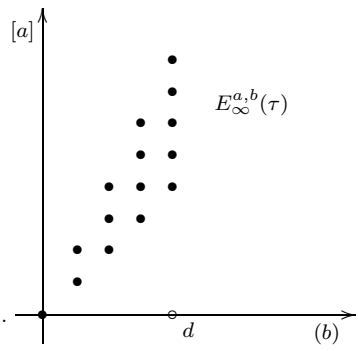
$$
\def\objectstyle{\scriptstyle}
\def\labelstyle{\scriptstyle}
\xymatrix @-1.5pc{
& & & & & & & & & & & & \\
& & & & & & & & & & & & \\
& & & & & \bub & & & & & & & \\
& & & & & \bub & & & & & & & \\
& & & &  \bub & \bub & & \ar@{}[r]^{E^{a,b}_{\infty}(\tau)} & & & & & \\
& & & & \bub & \bub & & & & & & & \\
& & & \bub & \bub & \bub & & & & & & & \\
& & & \bub & \bub & & & & & & & & \\
& & \bub & \bub & & & & & & & & & \\
& & \bub & & & & & & & & & & \\
\dub \ar@{->}[rrrrrrrrrrr]_(0.9){(b)} & \bub & & & & \wub
\ar@{}[r]_{d} 
& & & & & & & \\
& \dub \ar@{->}[uuuuuuuuuuu]^(0.9){[a]} & & & & & & & & & & & \\
}
$$

Note that the $\tau$-spectral sequence makes sense for
arbitrary motives. In particular, for Chow motives. 
If a Chow motive $M$ is a direct summand of a motive of a surface, then this sequence degenerates in the $E_2$-term and if, moreover, $k=\kbar$ is algebraically closed, then it degenerates already in the $E_1$-term.
Since singular cohomology of $M$ are finite-dimensional
vector spaces, this implies that the same holds about
$\km$-motivic cohomology, in this case.

Returning to our Godeaux motive $M$, we conclude (taking into account the action of Bockstein) that the Hodge half-diamond for $M_{\cc}$ looks as:
$$
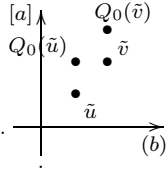
\def\objectstyle{\scriptstyle}
\def\labelstyle{\scriptstyle}
\xymatrix @-1.5pc{
& & & & & & & & \\
& & & \bub \ar@{}[r]^{Q_0(\tilde{v})}& & & & & \\
& & \bub \ar@{}[u]^{Q_0(\tilde{u})}& 
\bub \ar@{}[r]^{\tilde{v}} & & & & & \\
& & \bub \ar@{}[d]^{\tilde{u}} & & & & & & \\
\dub \ar@{->}[rrrrr]_(0.9){(b)} & & & & & & & & \\
& \dub \ar@{->}[uuuuu]^(0.9){[a]} & & & & & & & \\
}
$$
where each $\bullet$ corresponds to a 1-dimensional
$\zz/5$-space spanned by the respective element and 
elements $\tilde{u}$ and $\tilde{v}$ project to ${\mathbf u}$ and
${\mathbf v}\in \HH^*(M(\cc),\zz/5)$.
Taking into account that Tate-motives contribute only 
to the slope$=2$ part $a=2b$ of the half-diamond, 
we see
that the diagonal part $a=b$ of the $E_1$-term for the
Godeaux surface $X$ itself is $1\cdot\zz/5\oplus \tilde{u}\cdot\zz/5\oplus \tilde{v}\cdot\zz/5$. This is exactly the unramified
cohomology $\HH^*_{nr}(\cc(X)/\cc,\zz/5)$ of the generic
point of $X$ over $\cc$. So, $\tilde{u}$ is represented by some
unramified element of $K^M_1(\cc(X))/5$
and $\tilde{v}$ by some unramified element of $K^M_2(\cc(X))/5$. Since $\km$-motivic cohomology fits into the exact sequence:
$$
0\row\op{Coker}(\moco{a}{b-1}\stackrel{\tau}{\row}
\moco{a}{b})\row\moco{a}{b}(-,\km)\row
\kker(\moco{a+1}{b-1}\stackrel{\tau}{\row}\moco{a+1}{b})\row 0,
$$
where $\moco{*}{*'}=\moco{*}{*'}(-,\zz/5)$, and 
$\moco{2}{0}(X)=\moco{3}{1}(X)=0$, we see that
$\tilde{u}$ and $\tilde{v}$ lift to elements $u\in\moco{1}{1}(M_{\cc},\zz/5)$ and
$v\in\moco{2}{2}(M_{\cc},\zz/5)$ which are not divisible by $\tau$ and project to ${\mathbf u}$ and ${\mathbf v}$ in $\HH^*(M(\cc),\zz/5)$. 

From the description of the Hodge half-diamond for $M$
and the fact that Tate-motives don't have $\moco{3}{*}(-,\zz/5)$
over $\cc$, we see that the product $u\cdot v$ in
$\moco{*}{*'}(X_{\cc},\zz/5)$ is equal to $m\cdot\tau Q_0(v)$,
for some $m\in\zz/5$. In particular, $Q_0(u\cdot v)=0$.
Since $M(\cc)$ is self-dual with respect to $\Homi(-,\zz/5[4])$ in
$D^b(\zz/5)$, we can assume w.l.o.g. that
${\mathbf u}\cdot Q_0({\mathbf v})={\mathfrak e}\in \HH^4(X,\zz/5)$ the dual of the class of a point. In other words, the non-degenerate pairing
on $\HH^*(M,\zz/5)$ satisfies $\la {\mathbf u},Q_0({\mathbf v})\ra=1$. Then, since 
$Q_0({\mathbf u}\cdot{\mathbf v})=0$, we have 
$\la Q_0({\mathbf u}),{\mathbf v}\ra=1$.
Because multiplication by $\tau$ is injective
on $\moco{4}{*}(X_{\cc},\zz/5)$ (note that $M_{\cc}$
has no motivic cohomology on this row), we obtain that
$u\cdot Q_0(v)=\eps\cdot\tau$ and $Q_0(u)\cdot v=\eps\cdot\tau$
for the class of a $\cc$-point $\eps\in\moco{4}{2}(X_{\cc},\zz/5)$.

It follows from the proof of \cite[Proposition 2.3]{GO} that the identity morphism of $M$ is killed by multiplication by $5$. That is, $M$ is a torsion motive of exponent $5$. This is a consequence of the Beilinson-Lichtenbaum conjecture and the computation of singular cohomology of $M$.

Let now $k/\cc$ be some field extension.
Denote as $\mocod{i}$ the $i$-th diagonal
$\moco{i+*}{*}$ in motivic cohomology.
Since integral motivic cohomology of $M$
is killed by $5$, Bockstein $Q_0$ acts without cohomology on $\moco{*}{*'}(M_k,\zz/5)$. By the 
Beilinson-Lichtenbaum conjecture and Proposition 
\ref{Het-Xk}, we have:
$$
\mocod{0}(M_k,\zz/5)=u\cdot \HH^*_{et}\oplus\, 
v\cdot \HH^*_{et}\oplus\, \tau Q_0(u)\cdot \HH^*_{et}\oplus\,
\tau Q_0(v)\cdot \HH^*_{et}, 
$$
where $\HH^*_{et}=\HH^*_{et}(k,\zz/5)$. Multiplication by 
$\tau$ is injective on $\mocod{1}$ and so,
$$
\mocod{1}(M_k,\zz/5)=\tau^{-1}\cdot(u,v)\cdot A\oplus\, Q_0(u)\cdot \HH^*_{et}\oplus\,
Q_0(v)\cdot \HH^*_{et}, 
$$
where $A=\kker(\HH^*_{et}\oplus \HH^*_{et}
\stackrel{(\tilde{u},\tilde{v})}{\lrow} \HH^*_{et}(k(X)))$ (recall that $\moco{n}{n}(M_k,\km)$ is a direct
summand in $\HH^n_{nr}(k(X)/k,\zz/5)$, so a linear combination of $u$ and $v$ is divisible by $\tau$
if and only if the respective combination of $\tilde{u}$ and $\tilde{v}$ vanishes in unramified cohomology). 
The image of $Q_0:\mocod{0}(M_k,\zz/5)\row\mocod{1}(M_k,\zz/5)$ is 
$Q_0(u)\cdot \HH^*_{et}\oplus\, Q_0(v)\cdot \HH^*_{et}$,
and as $Q_0$ acts as an exact differential, it is injective on 
$\tau^{-1}\cdot(u,v)\cdot A$. On the other hand, the
map $Q_0:\mocod{1}(M_k,\zz/5)\row\mocod{2}(M_k,\zz/5)$
is surjective, since $Q_0$ is trivial on the second diagonal ($X$ is 2-dimensional). Hence,
$$
\mocod{2}(M_k,\zz/5)=\tau^{-1}\cdot Q_0(u,v)\cdot A.
$$
In particular, we see that multiplication by $\tau$ is injective on the second, and so, all the diagonals. Hence, all the differentials $d_r$ are zero in the $\tau$-spectral sequence of $M$, which implies that our spectral sequence degenerates at the 1-st page.

Considered for all finitely generated field extensions
$k/\cc$, $A$ forms a cycle module of Rost \cite{RCM},
which is a submodule of $\HH^*_{et}\oplus \HH^*_{et}$
(with appropriate degree shift). Note that this module has only a zero section over $\cc$, since there are no relations among $\tilde{u},\tilde{v}\in \HH^*_{nr}(\cc(X)/\cc,\zz/5)$ as $\cc$ has cohomological 
dimension zero. But below we will see that over appropriate field
extensions such relations will appear.

We can express the group of zero cycles on $M$.

\begin{cor}
 \label{zero-cycl-Godeaux}
 The group $\CH_0(M_k)=\CH^2(M_k)$ can be identified with
 $$
 \kker(\HH^2_{et}(k)\oplus \HH^1_{et}(k)
\stackrel{(\tilde{u},\tilde{v})}{\lrow} \HH^3_{et}(k(X))).
$$
\end{cor}

To get hold of our cycle module $A$, we will first establish certain orthogonality relations among its sections.

\begin{prop}
 \label{orthog-A}
 Let $k/\cc$ be some field extension and 
 $(\lambda_a,\mu_a),(\lambda_b,\mu_b)$ be two sections
 of $A$ over $k$. Then
 \begin{equation}
  \label{ortog}
 \mu_b\lambda_a+(-1)^{\ddeg(\lambda_a)\ddeg(\lambda_b)}\mu_a\lambda_b=0\in \HH^*_{et}(k).
 \end{equation}
\end{prop}

\begin{proof}
 We know that $u\cdot Q_0(v)=Q_0(u)\cdot v=\eps\cdot\tau\in\moco{4}{3}(X_k,\zz/5)$, where $\eps$ is the class of a $\cc$-point.
 Let $\alpha=(u\cdot \lambda_a+v\cdot\mu_a)$
 and $\beta=(Q_0(u)\cdot\lambda_b+Q_0(v)\cdot\mu_b)$

 Since $\tau^{-1}\cdot\alpha\in \tau^{-1}\cdot(u,v)\cdot A$ belongs to the 1-st diagonal, while $\tau^{-1}\cdot\beta$ belongs to the 2-nd diagonal in
 motivic cohomology of $M$, their product in $\moco{*}{*'}(X,\zz/5)$ must be zero. Then so is the product
 of $\alpha$ and
 $\beta$. But the push-forward 
 $\pi_*(\alpha\cdot\beta)$ of the latter product to the
 point is equal to $\tau\cdot((-1)^{\ddeg(\lambda_a)}\lambda_a\mu_b+\mu_a\lambda_b)$. Since multiplication
 by $\tau$ is injective in $\moco{*}{*'}(k,\zz/5)$, and elements anticommute with respect to the square degree (and $\ddeg(\lambda)=\ddeg(\mu)+1$), the triviality of this element is equivalent
 to the equation we need to prove.
 \Qed
\end{proof}

We have a map $M\stackrel{(u,v)}{\lrow}\zz/5(1)[1]\oplus\zz/5(2)[2]$. The dual of it under 
$\Homi(-,\zz(2)[4])$ will be
$\zz/5[1]\oplus\zz/5(1)[2]\stackrel{(v^{\vee},u^{\vee})}{\lrow}M$ (note that
$\Homi(\zz/5,\zz)=\zz/5[-1]$). Let $\pi:X\row\op{Spec}(\cc)$ be the projection. Since 
$\pi_*(u\cdot Q_0(v))=\pi_*(v\cdot Q_0(u))=\tau$
and $\pi_*(v\cdot Q_0(v))=\pi_*(u\cdot Q_0(u))=0$
in $\moco{*}{*'}(\cc,\zz/5)$, it follows that
$u\circ v^{\vee}=v\circ u^{\vee}=\cdot\tau$, while
$u\circ u^{\vee}=v\circ v^{\vee}=0$ and we obtain 
a self-dual composition
$$
\xymatrix{
\zz/5[1]\oplus\zz/5(1)[2] \ar[r]_(0.75){(v^{\vee},u^{\vee})}
\ar@/^1.5pc/ @{->}[rr]^{\cdot\tau}
&
M \ar[r]_(0.2){(u,v)}&
\zz/5(1)[1]\oplus\zz/5(2)[2]
}.
$$
Since $M$ is killed by multiplication by $5$, by the Beilinson-Lichtenbaum ``conjecture'', the map
$\moco{4}{2}(M\otimes M,\zz)\row \HH^4_{et}(M\otimes M,\zz/5)$ is injective. Since $(u,v)$ induces an isomorphism
on etale cohomology, our composition shows that
$M_{et}=(\zz/5)_{et}[1]\oplus(\zz/5)_{et}[2]$.

The above composition can be extended to a self-dual
octahedron:
$$
\def\objectstyle{\scriptstyle}
\def\labelstyle{\scriptstyle}
\xymatrix @-0.0pc{
 &{\ct}(-1) \ar[d] \ar[dr]^{\tau} & \\
 M_{\bullet} \ar[ru]^(0.4){[1]}  & M \ar[l] \ar[r] 
 \ar@{}[rd]|-(0.22){\star} \ar@{}[lu]|-(0.22){\star}& 
 {\ct}\ar[dl]^{[1]}\\
 & M^{\bullet} \ar[ul] \ar[u] &
}\hspace{8mm}
\xymatrix @-0.0pc{
 &{\ct}(-1) \ar[dr]^{\tau} & \\
 M_{\bullet} \ar[ru]^(0.4){[1]} \ar[r]^{[1]} & \ck \ar[u] \ar[d] 
 \ar@{}[ld]|-(0.22){\star} \ar@{}[ru]|-(0.22){\star}& 
 {\ct}\ar[dl]^{[1]} \ar[l]^{[1]}\\
 & M^{\bullet} \ar[ul] &
}
$$
where $\ct=\zz/5(1)[1]\oplus\zz/5(2)[2]$ and
$\ck=\km(1)\oplus\km(2)[1]$. All the objects here are compact.
The computation of motivic cohomology of $M$ above
gives that $\mocod{i}(M^{\bullet},\zz)=0$, for $i\neq 2$, while
$\mocod{2}(M^{\bullet},\zz)$ can be identified with
$\delta(\tau^{-1}\cdot(u,v)\cdot A)$, where
$\delta:\moco{*}{*'}(-,\zz/5)\row\moco{*+1}{*'}(-,\zz)$
is the connecting homomorphism. Under the pull-back
with respect to $\ck\row M^{\bullet}$ it is identified
with the submodule $(\ov{u},\ov{v})\cdot A$ of $\ov{u}\cdot \HH^*_{et}\oplus\,\ov{v}\cdot \HH^*_{et}=
\moco{*}{*'}(\ck,\zz)$, where $(\ov{u},\ov{v})$ is
$\delta$ of the canonical projection $\ck\row\ct(-1)$.
This implies that motivic cohomology of $M_{\bullet}$
is concentrated on the 3-rd diagonal and
$\mocod{3}(M_{\bullet},\zz)=(\ov{u},\ov{v})\cdot\left(\HH^*_{et}\oplus \HH^*_{et}/A\right)[1]$. The dual
calculations with $\mohod{i}$ - diagonals $\HH^{{\cal M}}_{i+*,*}$ in motivic homology show that we have an exact sequence:
$$
0\row\mohod{0}(M_{\bullet},\zz)\row\left(\ov{v}^{\vee}\cdot \HH^*_{et}\oplus\, \ov{u}^{\vee}\cdot \HH^*_{et}\right)
\row\mohod{-1}(M^{\bullet},\zz)\row 0,
$$
where $\mohod{0}(M_{\bullet},\zz)=
(\ov{v}^{\vee},\ov{u}^{\vee})\cdot A$, 
$\mohod{-1}(M^{\bullet},\zz)=(\ov{v}^{\vee},\ov{u}^{\vee})\cdot\left(\HH^*_{et}\oplus \HH^*_{et}/A\right)$
and all other diagonals are zero. Since motivic homology of $M_{\bullet}$ are concentrated on the zeroth diagonal, this object belongs to the heart of
the homotopy $t$-structure on $\dmk$ - see \cite{Deg}.
This heart can be identified with the abelian category of Rost cycle modules.
The cycle module of Rost corresponding to $M_{\bullet}$ is exactly $A$. Under this identification, $A$ is a graded 
submodule
of $\HH^*_{et}\la-1\ra\oplus \HH^*_{et}\la-2\ra$, where $\la m\ra$ indicates the shift in degrees. We will use this convention below.

We have an adjoint pair
$$
\xymatrix{
\dmkF{\zz} \ar@/_0.7pc/ @{->}[r]_{\nu^*} & 
\dmkF{\zz/5} \ar@/_0.7pc/ @{->}[l]_{\nu_*}
}
$$
where $\nu^*$ respects the $\otimes$ and $\nu_*$ satisfies
the projection formula. Since the cycle module $A$ is
5-torsion, it comes from $\dmkF{\zz/5}$ together with
the map $M_{\bullet}\row\ck[1]$. By obvious reasons,
the same is true about the map $\ck\row\ct(-1)$. Hence,
the motive $M$ itself is $\nu_*({\mathbf M})$, for some
motive ${\mathbf M}\in\dmkF{\zz/5}$.

Let $\Delta_M\in\moco{4}{2}(M\otimes M,\zz)$ be the
composition $M\otimes M\row M(X\times X)\stackrel{\Delta_X}{\row}\zz(2)[4]$. The etale version of it modulo $5$ is detected in the topological realisation and so, is equal to
$$
(\ov{\Delta}_M)_{et}=-u_{et}\times Q_0(v_{et})+v_{et}\times Q_0(u_{et})+Q_0(v_{et})\times u_{et}+Q_0(u_{et})\times v_{et}.
$$ 
Since $M$ is killed by $5$, the modulo
5 version $\ov{\Delta}_M$ is $Q_0$ of some element
from $\moco{3}{2}(M\otimes M,\zz/5)$. But the map from
Nisnevich to etale topology is injective on
$\mocod{1}(M\otimes M,\zz/5)$. Hence, 
$\ov{\Delta}_M=Q_0(\tau^{-1}(u\times v+v\times u))$
and so,
$\Delta_M=\delta(\tau^{-1}(u\times v+v\times u))$.
Since $\Delta_M$ is nothing else but the projector $(1-\rho)$ in Proposition \ref{M-dual}, and the remaining summands are Tate, we readily obtain
the following result.

\begin{prop}
 \label{delta-X}
 The class of the diagonal of the Godeaux surface can be expressed as
 $$
 \Delta_X=1\times\eps+\eps\times 1+\sum_i\alpha_i\times\beta_i+\delta(\tau^{-1}(u\times v+v\times u)),
 $$
 where $\alpha_i,\beta_i\in\CH^1(X)$, 
 $u\!\in\!\moco{1}{1}(X,\zz/5)$, $v\!\in\!\moco{2}{2}(X,\zz/5)$
 are elements introduced above and $\eps$ is the class
 of a $\cc$-point.
\end{prop}

Restricting to the generic point of $X$ (on one of the components of $X\times X$), and taking into account that $Q_0(u)$ and $Q_0(v)$ disappear, when restricted 
from $X$ to $\op{Spec}(\cc(X))$, we obtain that the class
$\eta$ of the generic point of $X$ in 
$\CH^2(X_{\cc(X)})/5$ is equal to
$$
\eta=\ov{\Delta}_X|_{\cc(X)}=\eps+
\tau^{-1}(Q_0(v)\cdot\tilde{u}+Q_0(u)\cdot\tilde{v}).
$$
Hence, $\tau^{-1}(Q_0(v)\cdot\tilde{u}+Q_0(u)\cdot\tilde{v})$ is the difference $\eta-\eps$ between the classes of the generic and complex points. This zero cycle of degree zero on $X$
represents a nonzero element in $\CH^2(M_{\cc(X)})$. Thus
$(\tilde{v},\tilde{u})\in A(\cc(X))$ is a non-trivial
section of our cycle module $A$ (indeed, we know that
the elements $\tilde{u},\tilde{v}\in \HH^*_{et}(\cc(X))$
are non-zero). It appears that this section generates
$A$ as a Rost cycle module.

\begin{prop}
 \label{A-gen}
 The Rost cycle module $A$ is generated by the section
 $(\tilde{v},\tilde{u})\in A(\cc(X))$.
\end{prop}

\begin{proof}
 From the above we know that the cycle module $A$ 
 can be identified with the second diagonal in
 motivic cohomology of $M$, which is
 $\mocod{2}(M,\zz/5)=\kker(\mocod{2}(X,\zz/5)\stackrel{\pi_*}{\row}\HH^*_{et})$, where $\pi$ is the 
 projection to the point. By \cite[Lemma 4.11]{Voe03} and \cite{RCM}, for a $d$-dimensional variety $X$,
 $\mocod{d}(X,\zz)$ coinsides with the Rost cycle 
 module of ``Chow groups with coefficients'' in Milnor's K-theory $\HH^d(X,\un{K}^M_*)$. And the
 projection $\mocod{d}(X,\zz)\row\mocod{d}(X,\zz/p)$
 is surjective. That means that, in our case, 
 $\mocod{2}(X_k,\zz/5)$ is additively generated by 
 elements of the form $\op{Tr}_{E/k}([q]\cdot\alpha)$,
 where $E/k$ is some finite extension, $q\in X(E)$ and
 $\alpha\in k^M_*(E)=K^M_*(E)/5$. Then the 
 $\kker(\mocod{2}(X,\zz/5)\stackrel{\pi_*}{\row}\HH^*_{et})$ is additively generated by the elements
 of the form $\op{Tr}_{E/k}(([q]-\eps)\cdot\alpha)$,
 where $\eps$ is the class of a $\cc$-point. But
 $[q]-\eps$ is a specialization of $\eta-\eps$. Hence,
 $\mocod{2}(M,\zz/5)$ as a Rost cycle module is 
 generated by $\eta-\eps=
 \tau^{-1}(Q_0(u)\cdot\tilde{v}+Q_0(v)\cdot\tilde{u})$.
 In other words, the cycle module $A$ is generated by 
 $(\tilde{v},\tilde{u})\in A(\cc(X))$.
 \Qed
\end{proof}

In particular, $\op{CH}_0(M_{\cc(X)})=\zz/5$ spanned by
$(\tilde{v},\tilde{u})$.
We also obtain:

\begin{prop}
 \label{A-max-ort}
 $A$ is a ``Lagrangian'' submodule of 
 $\HH^*_{et}\la-1\ra\oplus \HH^*_{et}\la-2\ra$, i.e. 
 a maximal submodule satisfying the orthogonality
 conditions $(\ref{ortog})$.
\end{prop}

\begin{proof}
We know that $(\tilde{v},\tilde{u})\in A(\cc(X))$.
On the other hand, for any field $k/\cc$,
$(\lambda,\mu)\in A(k)$ if and only if 
$\tilde{u}\cdot\lambda+\tilde{v}\cdot\mu=0\in \HH^*_{et}(k(X))$. Hence, $A(k)=(\HH^*_{et}(k)\la-1\ra\oplus \HH^*_{et}(k)\la-2\ra)
\cap A(k(X))^{\perp}$ and so, our submodule can't be enlarged without breaking the orthogonality conditions.
 \Qed
\end{proof}

\begin{rem}
 From Propositions \ref{A-gen} and
 \ref{A-max-ort} we see that $A$ is the unique submodule of $\HH^*_{et}\la-1\ra\oplus \HH^*_{et}\la-2\ra$ containing
 $(\tilde{v},\tilde{u})$ and satisfying the orthogonality conditions.
\end{rem}

\subsection{$p$-torsion motives of surfaces}

In this Subsection, unless specified otherwise, our ground field $k$ will be an algebraically closed field of characteristic zero.
Let us start with some general facts about torsion direct summands of surfaces. 

\begin{prop}
 \label{MN-kon}
 Let $M$ and $N$ be torsion direct summands in the motives of surfaces. Then the group $\Hom(M,N)$ is finite.
\end{prop}

\begin{proof}
 We will prove a stronger result: it is
 sufficient to assume that one of $M$ or $N$ is torsion.
 If $M$ and $N$ are direct summands in the motives of surfaces $X$ and $Y$, respectively, and $n$ kills $M$, or 
 $N$, then $\Hom(M,N)\hookrightarrow\CH^2(X\times Y)_{n-tors}$. But 
 $$
 \hm^{4,2}(X\times Y,\zz)_{n-tors}
 \stackrel{\delta}{\twoheadleftarrow}
 \hm^{3,2}(X\times Y,\zz/n)\hookrightarrow H^{3}_{et}(X\times Y,\zz/n), 
 $$
 where the last inclusion follows from 
 the Beilinson-Lichtenbaum ``conjecture''. 
 The rightmost group is finite, since $k$ is algebraically closed.
 \Qed
\end{proof}

This readily implies:

\begin{thm}
 \label{K-S}
 Let $k$ be a field of characteristic zero. Then torsion direct summands in the motives of surfaces over $k$ satisfy
 Krull-Schmidt principle.
\end{thm}

\begin{proof}
 By the result of S.Gille \cite{Gi14}, the restriction to the algebraic closure functor is conservative on the direct summands of the motives of surfaces. Then Proposition 
 \ref{MN-kon} implies that every torsion direct summand in 
 a motive of a surface decomposes into a direct sum of finitely many simple ones. Now, suppose that a simple torsion motive $M$ is a direct summand in $N\oplus L$.
 Then there are maps 
 $
 \xymatrix{
 N \ar@/^0.7pc/ @{->}[r]^{\psi_1} & 
 M \ar@/^0.7pc/ @{->}[l]^{\ffi_1} 
 \ar@/_0.7pc/ @{->}[r]_{\ffi_2}&
 L \ar@/_0.7pc/ @{->}[l]_{\psi_2}
 }
 $
 such that $\alpha=\psi_1\circ\ffi_1$ and
 $\beta=\psi_2\circ\ffi_2$ satisfy $\alpha+\beta=id_M$.
 Denote as $\ov{\alpha}$, etc. the restrictions to $\kbar$.
 Since $\End(\ov{M})$ is finite, some powers $\ov{\alpha}^n$ and 
 $\ov{\beta}^m$ are idempotents. If these are both zero, then
 $id_{\ov{M}}=(\ov{\alpha}+\ov{\beta})^{n+m}=0$ (note that $\alpha$ and 
 $\beta$ commute), which gives a contradiction.
 Hence, some power of $\ov{\alpha}$ or $\ov{\beta}$ is
 equal to a non-zero idempotent in $\End(\ov{M})$. Since 
 $\kker(\End(M)\row\End(\ov{M}))$ consists of nilpotents
 by the mentioned result of S.Gille, 
 and $M$ is indecomposable,
 by \cite[Lemma 2.4]{VY}, there exists some integral polynomial $\phi$ without constant term, whose value on 
 $\alpha$ or $\beta$ is equal to $id_M$. Then $M$ is a direct summand of one of $N$, or $L$. This shows that the decomposition into irreducibles is unique up to reordering.
 \Qed
\end{proof}

\begin{cor}
 \label{simple-M-con}
 In the situation of Theorem \ref{K-S}, 
 any indecomposable object 
 is a direct summand of the motive of a connected surface.
\end{cor}

\begin{prop}
 \label{bir-inv}
 In the situation of Theorem \ref{K-S},
 the maximal torsion direct summand of the motive of a surface $X$ is a (poly-)birational invariant of $X$.
\end{prop}

\begin{proof}
 It follows from the proof of Theorem \ref{K-S} that the maximal torsion direct summand of the motive of a surface is
 well-defined. If two surfaces $X$ and $Y$ have the same set of generic points, then one can get from $X$ to $Y$ by a series of blow ups and downs in smooth 0-dimensional centers. So, it is sufficient to consider a single blow-up
 $Y\row X$. But then $M(Y)=M(X)\oplus M(P)(1)[2]$, where 
 $P$ is the center. Since $M(P)(1)[2]$ doesn't contain
 torsion direct summands, any torsion direct summand of 
 $M(Y)$ should be a direct summand of $M(X)$, by the mentioned arguments of Theorem \ref{K-S}. The converse is obvious.
 \Qed
\end{proof}

We will try to classify the $p$-torsion motives appearing as direct summands in the motives of surfaces. 

Suppose $M$ is a torsion direct summand in the motive of some
(smooth projective) surface $X$ (possibly, disconnected), such that $p\cdot id_M=0$. 
From Proposition \ref{sm-Met-kbar} we know that $\HH^i_{et}(M,\zz/p)=0$, for $i=0$ and $i=4$, and since $p\cdot id_M=0$, the (mod $p$) Bockstein acts as an exact differential on $\HH^*_{et}(M,\zz/p)$. And the same is true about 
$M^{\vee}=\Homi(M,\zz(2)[4])$. 
This 
implies that the Hodge half-diamonds for $M$ and
$M^{\vee}$ look as:
$$
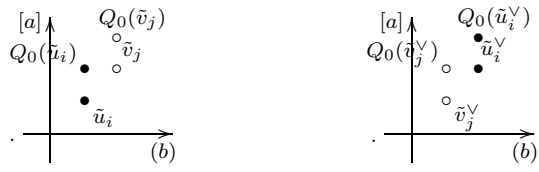
\def\objectstyle{\scriptstyle}
\def\labelstyle{\scriptstyle}
\xymatrix @-1.5pc{
& & & & & & & & \\
& & & \wub \ar@{}[r]^{Q_0(\tilde{v}_j)}& & & & & \\
& & \bub \ar@{}[u]^{Q_0(\tilde{u}_i)}& 
\wub \ar@{}[r]^{\tilde{v}_j} & & & & & \\
& & \bub \ar@{}[d]^{\tilde{u}_i} & & & & & & \\
\dub \ar@{->}[rrrrr]_(0.9){(b)} & & & & & & & & \\
& \dub \ar@{->}[uuuuu]^(0.9){[a]} & & & & & & & \\
}
\hspace{1cm}
\def\objectstyle{\scriptstyle}
\def\labelstyle{\scriptstyle}
\xymatrix @-1.5pc{
& & & & & & & & \\
& & & \bub \ar@{}[r]^{Q_0(\tilde{u}^{\vee}_i)}& & & & & \\
& & \wub \ar@{}[u]^{Q_0(\tilde{v}^{\vee}_j)}& 
\bub \ar@{}[r]^{\tilde{u}^{\vee}_i} & & & & & \\
& & \wub \ar@{}[d]^{\tilde{v}^{\vee}_j} & & & & & & \\
\dub \ar@{->}[rrrrr]_(0.9){(b)} & & & & & & & & \\
& \dub \ar@{->}[uuuuu]^(0.9){[a]} & & & & & & & \\
}
$$
where each $\bullet$ represents an $s$-dimensional vector space over $\ff_p$ spanned by classes:
$\tilde{u}_i$ for $\ddeg=(1)[1]$, $Q_0(\tilde{u}_i)$ for
$\ddeg=(1)[2]$, $\tilde{u}^{\vee}_i$ for $\ddeg=(2)[2]$, and
$Q_0(\tilde{u}^{\vee}_i)$ for $\ddeg=(2)[3]$,
and each $\circ$ represents a $t$-dimensional vector space over $\ff_p$ spanned by classes:
$\tilde{v}^{\vee}_j$ for $\ddeg=(1)[1]$, $Q_0(\tilde{v}^{\vee}_j)$ for
$\ddeg=(1)[2]$, $\tilde{v}_j$ for $\ddeg=(2)[2]$, and
$Q_0(\tilde{v}_j)$ for $\ddeg=(2)[3]$. The pairing
$$
\HH^a_{et}(M,\zz/p)\times \HH^c_{et}(M^{\vee},\zz/p)\row \HH^{a+c-4}_{et}(k,\zz/p)
$$
is given by 
\begin{equation*}
 \begin{split}
 &\la \tilde{u}_i,Q_0(\tilde{u}^{\vee}_i)\ra=
\la Q_0(\tilde{u}_i), \tilde{u}^{\vee}_i\ra=1, 
\hspace{2mm}\text{for}\,1\leq i\leq s\\
 &\la\tilde{v}_j,Q_0(\tilde{v}^{\vee}_j)\ra=
-\la Q_0(\tilde{v}_j),\tilde{v}^{\vee}_j\ra=1,
\hspace{2mm}\text{for}\,1\leq j\leq t,
 \end{split}
\end{equation*}
with all other
combinations being zero.
By the Beilinson-Lichtenbaum ``conjecture'', $\tilde{u}_i$
and $\tilde{v}_j$ can be lifted to elements 
$u_i\in\hm^{1,1}(M,\zz/p)$ and $v_j\in\hm^{2,2}(M,\zz/p)$,
while $\tilde{v}^{\vee}_j$ and $\tilde{u}^{\vee}_i$ can be lifted to elements 
$v^{\vee}_j\in\hm^{1,1}(M^{\vee},\zz/p)$ and $u^{\vee}_i\in\hm^{2,2}(M^{\vee},\zz/p)$,
such that $\la u_i,Q_0(u^{\vee}_i)\ra=\la Q_0(u_i),u^{\vee}_i\ra=\la v_j,Q_0(v^{\vee}_j\ra=-\la Q_0(v_j),v^{\vee}_j\ra=\tau$
(zero otherwise), for the pairing
$$
\hm^{a,b}(M,\zz/p)\times\hm^{c,d}(M^{\vee},\zz/p)\row\hm^{a+c-4,b+d-2}(k,\zz/p).
$$
In other words, if we denote $\ct_A=\zz/p(1)[1]^{\oplus s}\oplus\zz/p(2)[2]^{\oplus t}$ and $\ct_B=\zz/p(1)[1]^{\oplus t}\oplus\zz/p(2)[2]^{\oplus s}$, then the compositions
\begin{equation}
\label{posled-t-m}
\ct_A(-1)\stackrel{(u_i^{\vee},v_j^{\vee})}{\lrow} M
\stackrel{(u_i,v_j)}{\lrow}\ct_A
\hspace{2cm}
\ct_B(-1)\stackrel{(v_j,u_i)}{\lrow} M^{\vee}
\stackrel{(v^{\vee}_j,u^{\vee}_i)}{\lrow}\ct_B
\end{equation}
coincide with the multiplication by $\tau$.
We can complete these to dual octahedra:
\begin{equation}
\label{oktaedry}
\begin{split}
\def\objectstyle{\scriptstyle}
\def\labelstyle{\scriptstyle}
\xymatrix @-0.0pc{
 &{\ct_A}(-1) \ar[dr]^{\tau} & \\
 M_{\bullet}(A) \ar[ru]^(0.4){[1]} \ar[r]^{[1]} & \ck_A \ar[u] \ar[d] 
 \ar@{}[ld]|-(0.22){\star} \ar@{}[ru]|-(0.22){\star}& 
 {\ct_A}\ar[dl]^{[1]} \ar[l]^{[1]}\\
 & M^{\bullet}(A) \ar[ul]^{\gamma_A} &
}\hspace{8mm}
\xymatrix @-0.0pc{
 &{\ct_A}(-1) \ar[d]^{f} \ar[dr]^{\tau} & \\
 M_{\bullet}(A) \ar[ru]^(0.4){[1]}  & M \ar[l] \ar[r]^{g} 
 \ar@{}[rd]|-(0.22){\star} \ar@{}[lu]|-(0.22){\star}& 
 {\ct_A}\ar[dl]^{[1]}\\
 & M^{\bullet}(A) \ar[ul]^{\gamma_A} \ar[u] &
}\\
\text{and}\\
\def\objectstyle{\scriptstyle}
\def\labelstyle{\scriptstyle}
\xymatrix @-0.0pc{
 &{\ct_B}(-1) \ar[dr]^{\tau} & \\
 M_{\bullet}(B) \ar[ru]^(0.4){[1]} \ar[r]^{[1]} & \ck_B \ar[u] \ar[d] 
 \ar@{}[ld]|-(0.22){\star} \ar@{}[ru]|-(0.22){\star}& 
 {\ct_B}\ar[dl]^{[1]} \ar[l]^{[1]}\\
 & M^{\bullet}(B) \ar[ul]^{\gamma_B} &
}\hspace{8mm}
\xymatrix @-0.0pc{
 &{\ct_B}(-1) \ar[d]^{g^{\vee}} \ar[dr]^{\tau} & \\
 M_{\bullet}(B) \ar[ru]^(0.4){[1]}  & M^{\vee} \ar[l] \ar[r]^{f^{\vee}} 
 \ar@{}[rd]|-(0.22){\star} \ar@{}[lu]|-(0.22){\star}& 
 {\ct_B}\ar[dl]^{[1]}\\
 & M^{\bullet}(B) \ar[ul]^{\gamma_B} \ar[u] &
},
\end{split}
\end{equation}
where $\ck_A=\km(1)^{\oplus s}\oplus\km(2)[1]^{\oplus t}$,
$\ck_B=\km(1)^{\oplus t}\oplus\km(2)[1]^{\oplus s}$ and
$\km=\op{Cone}(\zz/p(-1)\stackrel{\cdot\tau}{\row}\zz/p)$.
Note that $\km$ belongs to the heart of the homotopic $t$-structure \cite{Deg} and represents there the Rost
cycle module $k^M_*=K^M_*/p$. In particular, we get dual
to each other distinguished triangles in $\dmgmk$:
\begin{equation}
\label{triangle}
\xymatrix{
& \ck_A \ar[dr]^{\beta_A} \ar@{}[d]|-(0.55){\star} & \\
M_{\bullet}(A) \ar[ur]^(0.4){\alpha_A} 
\ar@{}[ur]^(0.7){[1]}
& & M^{\bullet}(A) \ar[ll]^{\gamma_A}
}\hspace{5mm}\text{and}\hspace{5mm}
\xymatrix{
& \ck_B \ar[dr]^{\beta_B} \ar@{}[d]|-(0.55){\star} & \\
M_{\bullet}(B) \ar[ur]^(0.4){\alpha_B} 
\ar@{}[ur]^(0.7){[1]}
& & M^{\bullet}(B) \ar[ll]^{\gamma_B}
}.
\end{equation}
We will use the standard notations 
$\moco{i}{j}(N,\zz):=\Hom_{DM(k)}(N,\zz(j)[i])$ for {\it motivic cohomology}, respectively, $\moho{i}{j}(N,\zz):=\Hom_{DM(k)}(\zz(j)[i],N)$ for {\it motivic homology}.

In view of (\ref{posled-t-m}), $M_{et}=(\zz/p)_{et}[1]^{\oplus s}\oplus(\zz/p)_{et}[2]^{\oplus t}$ and $f_{et}=id_{M_{et}}=g_{et}$. By the 
Beilinson-Lichtenbaum ``conjecture'', $g^*$ is an isomorphism
on motivic diagonals $\mocod{i}=\hm^{i+*,*}$ with $i\leq 1$,
and so, $\mocod{i}(M^{\bullet}(A),\zz)=0$, for $i\neq 2$, while
$\mocod{2}(M^{\bullet}(A),\zz)=\mocod{2}(M,\zz)$, which can be identified with $\mohod{0}(M^{\vee},\zz)$. By the same reason, $f^*$ is injective on diagonals $\mocod{\leq 2}$, and so, on all of them (as $\ddim(X)=2$) and $\mocod{i}(M_{\bullet}(A),\zz)=0$, for $i\neq 3$, while
$\mocod{3}(M_{\bullet}(A),\zz)$ can be identified with $H_B/\mohod{0}(M^{\vee},\zz)$, where
$H_B=\HH^*_{et}\la-1\ra^{\oplus t}\oplus \HH^*_{et}\la-2\ra^{\oplus s}$ and
$\HH^*_{et}=\HH^*_{et}(k,\zz/p)$. In particular, $\gamma_A^*=0$. 

The same considerations apply to the second triangle, $M^{\bullet}(B)$ and $M_{\bullet}(B)$. In particular,
$\mocod{2}(M^{\bullet}(B),\zz)$ can be identified with
$\mohod{0}(M,\zz)$ and $\mocod{3}(M_{\bullet}(B),\zz)$ - with
$H_A/\mohod{0}(M,\zz)$, where $H_A=\HH^*_{et}\la-1\ra^{\oplus s}\oplus \HH^*_{et}\la-2\ra^{\oplus t}$.
By duality, $(\gamma_{A,B})_*=0$ as well.
This shows that motivic homology of $M_{\bullet}(A)$ and
$M_{\bullet}(B)$ is concentrated on the $0$-th diagonal and so, these objects belong to the
heart of the homotopic $t$-structure. There it is represented by some Rost cycle submodules $A\subset H_A$ and
$B\subset H_B$. Since 
$A=\mohod{0}(M_{\bullet}(A),\zz)=\mohod{0}(M,\zz)$ and
$B=\mohod{0}(M_{\bullet}(B),\zz)=\mohod{0}(M^{\vee},\zz)$, we obtain that $\moho{i}{i}(M_{\bullet}(A),\zz)=0$ and 
$\moho{i}{i}(M_{\bullet}(B),\zz)=0$, for $i>0$, that is, $A$ and $B$ are concentrated in non-negative degrees (recall, that the generators of the modules $H_A$ and $H_B$ are in
degrees $-1$ and $-2$).

\medskip

Conversely, let 
$\ck_A=\km(1)^{\oplus s}\oplus\km(2)[1]^{\oplus t}$,
$\ck_B=\km(1)^{\oplus t}\oplus\km(2)[1]^{\oplus s}$,
and suppose we have
a pair of dual with respect to $\Homi(-,\zz(2)[4])$
distinguished triangles (\ref{triangle}) 
of geometric motives, where $\gamma^*_{A,B}$ is zero on
$\moco{*}{*'}$ ($\Leftrightarrow(\gamma_{A,B})_*=0$) and 
$\moho{i}{i}(M_{\bullet}(A),\zz)=0$, 
$\moho{i}{i}(M_{\bullet}(B),\zz)=0$, for $i>0$ . 

The (integral) motivic homology of $\ck_A$ is
a free module over $\HH^*_{et}$ of rank $(s+t)$ 
with generators $\ov{v}_j^{\vee}$ and $\ov{u}_i^{\vee}$ in degrees $(1)$ and $(2)[1]$ which
via $\Homi(\ck_A,\zz(2)[4])=\ck_B[1]$ can be identified with the 
motivic cohomology of $\ck_B$, which is a similar free module
of rank $(s+t)$ with generators $\ov{u}_i$ and $\ov{v}_j$ in degree $[2]$ and $(1)[3]$, where we identify
$U_i:=\ov{u}_i=\ov{u}_i^{\vee}$, $V_j:=\ov{v}_j=\ov{v}_j^{\vee}$. Similarly, motivic homology of $\ck_B$ can be identified with the motivic cohomology of $\ck_A$.
Since $\gamma^*_{A,B}=0$, 
for $M_{\bullet}=M_{\bullet}(A,B)$ and
$M^{\bullet}=M^{\bullet}(A,B)$ we obtain:
\begin{equation*}
\begin{split}
&\mocod{l}(M^{\bullet},\zz)=0,\hspace{2mm}\text{for}\hspace{2mm} l\neq 2,\hspace{3mm} 
\mocod{l}(M_{\bullet},\zz)=0,\hspace{2mm}\text{for}
\hspace{2mm}l\neq 3;\\
&\mohod{l}(M^{\bullet},\zz)=0,\hspace{2mm}\text{for} 
\hspace{2mm}l\neq -1,\hspace{3mm}
\mohod{l}(M_{\bullet},\zz)=0,\hspace{2mm}\text{for} 
\hspace{2mm}l\neq 0, 
\end{split}
\end{equation*}
and
via the above identification, 
$\mohod{0}(M_{\bullet}(A),\zz)=\mocod{2}(M^{\bullet}(B),\zz)$,
$\mocod{3}(M_{\bullet}(A),\zz)=\mohod{-1}(M^{\bullet}(B),\zz)$ (and similar with $A$ and $B$ swapped) which fit into the exact sequences
\begin{equation}
\label{ho-co-Mlow}
\begin{split}
&0\row\mohod{0}(M_{\bullet}(A),\zz)\stackrel{(\alpha_A)_*}{\lrow}
H_A\stackrel{\alpha_B^*}{\lrow}
\mocod{3}(M_{\bullet}(B),\zz)\row 0;\\
&0\row\mohod{0}(M_{\bullet}(B),\zz)\stackrel{(\alpha_B)_*}{\lrow}
H_B\stackrel{\alpha_A^*}{\lrow}
\mocod{3}(M_{\bullet}(A),\zz)\row 0.
\end{split}
\end{equation}
Thus, $M_{\bullet}(A)$ and $M_{\bullet}(B)$ belong to the heart of the homotopic $t$-structure and correspond to Rost cycle submodules $A\subset H_A$ and $B\subset H_B$. Moreover, the dual of these should
be given by the respective quotient cycle modules
(up to appropriate shift). 

Combining our distinguished triangles with the following ones 
$$
\ct_{A,B}(-1)\stackrel{\cdot\tau}\lrow\ct_{A,B}\lrow\ck_{A,B}[1]\lrow\ct_{A,B}(-1)[1], 
$$
where 
$\ct_A=\zz/p(1)[1]^{\oplus s}\oplus\zz/p(2)[2]^{\oplus t}$ and 
$\ct_B=\zz/p(1)[1]^{\oplus t}\oplus\zz/p(2)[2]^{\oplus s}$, we get a pair of dual w.r.to $\Homi(-,\zz(2)[4])$ octahedra
(\ref{oktaedry}) consisting of compact objects. 
In particular, motive $M$ appears.

By the very construction, $M$ 
comes from
$\dmkF{\zz/p}$ (since it is so for the morphism 
$M_{\bullet}(A)\row\ct_A(-1)$). 
In particular, it is $p$-torsion.
Since $\mocod{i}(\ct_A,\zz)=0$, for $i>1$ and
$\mocod{i}(M^{\bullet}(A),\zz)=0$, for $i\neq 2$, we obtain that $\mocod{i}(M,\zz)=0$, for $i>2$ and so,
the map $M\lrow M_{\bullet}(A)$ is
zero on cohomology. By the same reason, the map
$\ct\stackrel{[1]}{\lrow}M^{\bullet}$ is also trivial
on cohomology.
We have the pairing 
$\la -,-\ra$ 
$$
\moco{a}{b}(M,\zz/p)\otimes\moco{c}{d}(M^{\vee},\zz/p)\lrow
\moco{a+c-4}{b+d-2}(k,\zz/p)
$$
which is $\moco{*}{*'}(k,\zz/p)$-linear with
respect to the left action on the left factor
and the right action on the right one. It is also
compatible with the pairing
$$
\moco{a}{b}(\ct_A(-1),\zz/p)\otimes\moco{c}{d}(\ct_B,\zz/p)\lrow
\moco{a+c-4}{b+d-2}(k,\zz/p)
$$
in the sense that $\la f^*(x),y\ra=\la x,(f^{\vee})^*(y)\ra$,
for $x\in\moco{*}{*'}(M,\zz/p)$, $y\in\moco{*}{*'}(\ct_B,\zz/p)$. 
The map $f^{\vee}$ is given by $M^{\vee}\stackrel{(v^{\vee}_j,u^{\vee}_i)}{\lrow}
\ct_B$ and so, we get:
\begin{equation}
\label{spar}
\la u_i,Q_0(u^{\vee}_i)\ra=\la Q_0(u_i),u^{\vee}_i\ra=\la v_j,Q_0(v^{\vee}_j)\ra=-\la Q_0(v_j),v^{\vee}_j\ra=\tau,
\end{equation}
with all other combinations equal to zero.\\

Because $\moho{i}{i}(M_{\bullet}(A),\zz)=0$, for $i>0$,  we obtain that 
$\moho{i}{j}(M,\zz)=0$, for 
$i<j$, for $i<2j$ and for $i>2$. 
And the same is true about $M^{\vee}$. By duality,
$\moco{i}{j}(M,\zz)=0$ for $i-j>2$, for $i>2j$ and for $i<0$.
This holds over any extension $K/k$.

It follows from Proposition \ref{Chow-crit} that our motive $M$ is a Chow motive. Then Proposition \ref{N-r-eff} implies
that $M$ is a direct summand in the
motive $M(X)$ of a smooth projective (possibly reducible)
surface $X$.\\

Thus, we have proven the following result.

\begin{thm}
 \label{tor-sur-oto}
 The assignment $M\mapsto A=\mohod{0}(M,\zz),\,B=\mohod{0}(M^{\vee},\zz)$ defines a $1$-to-$1$ correspondence between isomorphism classes of
 \begin{itemize}
  \item[$(1)$] $p$-torsion direct summands in the motives of smooth projective surfaces (possibly, disconnected); \hspace{1cm} and
  \item[$(2)$] Pairs of Rost submodules $A\stackrel{(\alpha_A)_*}{\lrow} H_A$, $B\stackrel{(\alpha_B)_*}{\lrow} H_B$,
  where 
  $H_A=\HH^*_{et}\la-1\ra^{\oplus s}\oplus 
  \HH^*_{et}\la-2\ra^{\oplus t}$,
  $H_B=\HH^*_{et}\la-1\ra^{\oplus t}\oplus 
  \HH^*_{et}\la-2\ra^{\oplus s}$, for some $s$ and $t$, such that:
  \begin{itemize}
  \item[$(a)$] The respective objects $M_{\bullet}(A)$  and
  $M_{\bullet}(B)$ of the heart of the homotopic $t$-structure
  are compact.
  \item[$(b)$] $A^j=B^j=0$, for $j<0$;
  \item[$(c)$] The respective triangles $(\ref{triangle})$
  are dual 
  with respect to $\Homi(-,\zz(2)[4])$, namely, $\gamma_B=\gamma_A^{\vee}$
  (here $\ck_A=\km(1)^{\oplus s}\oplus\km(2)[1]^{\oplus t}$
  and $\ck_B=\km(1)^{\oplus t}\oplus\km(2)[1]^{\oplus s}$).
  \end{itemize} 
\end{itemize}
\end{thm}

\begin{rem}
 \label{AB-rem}
 Of course, in the above Theorem, the Rost submodule $B$ is determined by $A$, and vice-versa. So, the result can be formulated in terms of $A$ only. But then the conditions will be less natural. The motive $M$ will be self-dual exactly when $s=t$ and $A=B$. 
 \Red
\end{rem}

Furthermore, such a pair $(A,B)$ of Rost submodules enjoys various additional properties.

\begin{prop}
 \label{ort-cond-A-gen}
 For any field extension $F/k$ and any sections $(\mu^i_a,\lambda^j_a)$ and $(\mu^j_b,\lambda^i_b)$ of
 $A(F)$ and $B(F)$, respectively, the following orthogonality relation holds:
  \begin{equation}
 \label{ortog-p}
 \sum_j(-1)^{\ddeg(\lambda_a)\ddeg(\lambda_b)}\mu^j_b\lambda^j_a+\sum_i\mu^i_a\lambda^i_b=0\in H^*_{et}(F,\zz/p).
\end{equation}
\end{prop}

\begin{proof}
Because $M$ comes from $\dmkF{\zz/p}$,
the Bockstein $Q_0$ acts as an exact differential on
$\moco{*}{*'}(M^{\vee},\zz/p)$. 
Since $\mocod{2}(M_F,\zz)\row
\mocod{2}(M^{\bullet}(A)_F,\zz)$ 
and $\mocod{2}(M^{\vee}_F,\zz)\row\mocod{2}(M^{\bullet}(B)_F,\zz)$
are isomorphisms, 
$(\mu^i_a,\lambda^j_a)\in A(F)$ corresponds to
$\tilde{a}=\tau^{-1}(\sum_j v^{\vee}_j\cdot\lambda^j_a+
\sum_i u^{\vee}_i\cdot\mu^i_a)\in\mocod{1}(M^{\vee}_F,\zz/p)$ and
$(\mu^j_b,\lambda^i_b)\in B(F)$ corresponds to
$Q_0(\tilde{b})=\tau^{-1}(\sum_i Q_0(u_i)\cdot\lambda^i_b+\sum_j 
Q_0(v_j)\cdot\mu^j_b)\in
\mocod{2}(M_F,\zz/p)$.  
Since $\la Q_0(\tilde{b}),\tilde{a}\ra\in
\mocod{1}(F,\zz/p)=0$, the pairing $\la
\tau Q_0(\tilde{b}),\tau \tilde{a}\ra$ is also zero. Using 
(\ref{spar}), we obtain the orthogonality relation.
 \Qed
\end{proof}

Let $M=(X,\rho)$ and $X=\coprod_r X_r$ be the decomposition of $X$ into a disjoint union of its connected components.
Since the $E_1$-term of the $tau$-spectral sequence
coincides with the $E_2$-term of the $coniveau$-spectral sequence, we have an embedding 
\begin{equation}
\label{vlozh-nr}
\displaystyle\frac{\mocod{0}(M_F,\zz/p)}
{\tau\cdot\mocod{1}(M_F,\zz/p)}\subset
\oplus_r \HH^*_{nr}(F(X_r)/F,\zz/p)
\end{equation}
of $\HH^*_{et}(F)$-modules, for any $F/k$.

As was discussed above, 
$M_{et}=(\zz/p)_{et}[1]^{\oplus s}\oplus(\zz/p)_{et}[2]^{\oplus t}$ and for the modulo $p$ version $\ov{M}_{et}$ of it, we have:
$$
(\overline{\Delta}_M)_{et}=\sum_i[-(u_i)_{et}\times Q_0((u^{\vee}_i)_{et})+
Q_0((u_i)_{et})\times (u^{\vee}_i)_{et}]+\sum_j[(v_j)_{et}\times Q_0((v^{\vee}_j)_{et})+
Q_0((v_j)_{et})\times (v^{\vee}_j)_{et}]
$$
and so,
$$
(\Delta_M)_{et}=\delta\left(\sum_i(u_i)_{et}\times (u^{\vee}_i)_{et}+\sum_j(v_j)_{et}\times (v^{\vee}_j)_{et}\right).
$$
Since $Q_0(u^{\vee}_i)$ and $Q_0(v^{\vee}_j)$ are contained in
$\mocod{1}(X,\zz/p)$, these have a positive co-dimension
of support and so,
$$
(\ov{\Delta}_M|_{k(X_r)})_{et}=\sum_iQ_0((u_i)_{et})\cdot (u^{\vee}_i)_{k(X_r)}+\sum_j
Q_0((v_j)_{et})\cdot (v^{\vee}_j)_{k(X_r)}.
$$
Since the map $M\row M_{\bullet}(A)$ is zero on cohomology,
the multiplication by $\tau$ is injective on $\moco{*}{*'}(M_E,\zz/p)$, for any field extension $E/k$ (this implies,
in particular, that the $tau$-spectral sequence for $M$
degenerates on the $E_1$-page). From the injectivity of the
map $\mocod{2}(M,\zz/p)\stackrel{\cdot\tau}{\lrow}
\mocod{1}(M,\zz/p)$ we obtain that
\begin{equation}
\label{gen-pt-M}
\Delta_M|_{k(X_r)}=\delta(\tau^{-1}\left(\sum_i u_i\cdot (u^{\vee}_i)_{k(X_r)}+\sum_j v_j\cdot (v^{\vee}_j)_{k(X_r)}\right)).
\end{equation}
Since $\Delta_M|_{k(X_r)}\in\mocod{2}(M_{k(X_r)},\zz)$, we 
obtain that $((v^{\vee}_j)_{k(X_r)},(u^{\vee}_i)_{k(X_r)})\in B(k(X_r))$.
In the same way, $((u_i)_{k(X_r)},(v_j)_{k(X_r)})\in A(k(X_r))$.
Moreover, in the light of the embedding (\ref{vlozh-nr}),
we obtain:

\begin{prop}
 \label{samo-ort-gen}
For any extension $K/k$, 
$$
(\mu^i_a,\lambda^j_a)\in
A(K)\Leftrightarrow\sum_j(v^{\vee}_j)_{K(X_r)}\cdot\lambda^j_a+\sum_i(u^{\vee}_i)_{K(X_r)}\cdot\mu^i_a=0\in \HH^*_{nr}(K(X_r)/K,\zz/p),
\forall r.
$$
In other words, $A(K)$ consists exactly of those elements
of $H_A(K)$ which are orthogonal
to $((v^{\vee}_j)_{K(X_r)},(u^{\vee}_i)_{K(X_r)})$ in $H_A(K(X_r))$, for all $r$, in terms of the relations
$(\ref{ortog-p})$. Similarly,
$B(K)$ consists exactly of those elements of $H_B(K)$ which
are orthogonal to $((u_i)_{k(X_r)},(v_j)_{k(X_r)})$ in $H_B(K(X_r))$, for all $r$.
\end{prop}

\begin{cor}
 \label{max-self-ort-gen}
 $A=B^{\perp}$ and $B=A^{\perp}$ as Rost cycle modules, in terms of the relation $(\ref{ortog-p})$.
\end{cor}

By \cite[Lemma 4.11]{Voe03} and \cite{RCM}, 
 $\mocod{2}(X,\zz)$ coincides with the Rost cycle 
 module of ``Chow groups with coefficients'' in Milnor's K-theory $\HH^2(X,\un{K}^M_*)$. Thus,
 $\mocod{2}(M_K,\zz)$ is additively generated by 
 elements of the form $\op{Tr}_{E/K}(\rho_M([q]\cdot\alpha))$,
 where $E/K$ is some finite extension, $q\in X(E)$,
 $\alpha\in K^M_*(E)$ and $\rho_M$ is the 
 projector defining $M$. Since $q$ is a specialization of 
 one of the
 generic points $\op{Spec}(E(X_r))$ of $X_E$, we get from 
 (\ref{gen-pt-M}) that
 $$
 \rho_M([q]\cdot\alpha)=\rho_M([q])\cdot\alpha=
 \delta(\tau^{-1}\Big(\sum_i u_i\cdot (u^{\vee}_i)_{E(X_r)}+\sum_j v_j\cdot (v^{\vee}_j)_{E(X_r)}\Big))(q)\cdot\alpha.
 $$
 Since $\mocod{2}(M,\zz)$ can be identified with the Rost cycle module $B$, we obtain:
 
\begin{prop}
 \label{generators-A}
 The Rost cycle module $B$
 is generated by elements 
 $((u^{\vee}_i)_{k(X_r)},(v^{\vee}_j)_{k(X_r)})\in B(k(X_r))$, numbered by
 the connected components of $X$. Similarly,
 the Rost cycle module $A$ is generated by elements
 $((u_i)_{k(X_r)},(v_j)_{k(X_r)})\in A(k(X_r))$, numbered by
 the same components.
\end{prop}

\begin{rem}
 \label{sim-gen-one}
 In the light of Corollary \ref{simple-M-con}, for an irreducible $p$-torsion motive $M$, the respective Rost
 cycle modules $A$ and $B$ are principal (generated by a single generator).
 \Red
\end{rem}

For a smooth projective curve $C/k$ having $m$ connected components, $M(C)=\zz^{\oplus m}\oplus \ov{M}(C)\oplus
\zz(1)[2]^{\oplus m}$, where $\ov{M}(C)$ is the {\it reduced motive} of $C$ (recall that our field $k$ is algebraically closed and so, every variety has a $k$-rational point).

Using arguments completely parallel to the proof of Theorem \ref{tor-sur-oto}, one gets:

\begin{thm}
 \label{curv}
 The assignment $N\mapsto C=\mohod{0}(N,\zz/p),\,D=\mohod{0}(N^{\vee},\zz/p)$ defines a $1$-to-$1$ correspondence between isomorphism classes of
 \begin{itemize}
  \item[$(1)$] Direct summands in the reduced motives of smooth projective (possibly, disconnected) curves in $Chow(k,\zz/p)$; \hspace{1cm} and
  \item[$(2)$] Pairs of Rost submodules 
  $C\stackrel{(\alpha_C)_*}{\lrow} H_C$, 
  $D\stackrel{(\alpha_D)_*}{\lrow} H_D$,
  where 
  $H_C=H_D=\HH^*_{et}\la-1\ra^{\oplus r}$, for some $r$, 
  such that:
  \begin{itemize}
  \item[$(a)$] The respective objects $N_{\bullet}(C)$  and
  $N_{\bullet}(D)$ of the heart of the homotopic $t$-structure (on $\dmkF{\zz/p}$)
  are compact.
  \item[$(b)$] $C^j=D^j=0$, for $j<0$;
  \item[$(c)$] The respective triangles 
  (with $\ck_C=\ck_D=k_M(1)^{\oplus r}$)
  \begin{equation*}
\xymatrix{
& \ck_C \ar[dr]^{\beta_C} \ar@{}[d]|-(0.55){\star} & \\
N_{\bullet}(C) \ar[ur]^(0.4){\alpha_C} 
\ar@{}[ur]^(0.7){[1]}
& & N^{\bullet}(C) \ar[ll]^{\gamma_C}
}\hspace{5mm}\text{and}\hspace{5mm}
\xymatrix{
& \ck_D \ar[dr]^{\beta_D} \ar@{}[d]|-(0.55){\star} & \\
N_{\bullet}(D) \ar[ur]^(0.4){\alpha_D} 
\ar@{}[ur]^(0.7){[1]}
& & N^{\bullet}(D) \ar[ll]^{\gamma_D}
}.
\end{equation*}
  are dual w.r. to $\Homi_{\dmkF{\zz/p}}(-,\zz/p(1)[2])$,
  namely, $\gamma_D=\gamma_C^{\vee}$.
  \end{itemize}
  \end{itemize}
\end{thm}

This result permits to draw some conclusions about $p$-torsion motives of surfaces.

\begin{prop}
 \label{s-t-zero}
 In the notations of Theorem \ref{tor-sur-oto},
 for nonzero $M$, both $s$ and $t$ are non-zero.
\end{prop}

\begin{proof}
 Suppose, $s\cdot t=0$. Changing $M$ by $M^{\vee}$, if necessary, we can assume that $t=0$. Then, 
 $A\subset H_A=\HH^*_{et}\la-1\ra^{\oplus s}$, $B\subset H_B=\HH^*_{et}\la-2\ra^{\oplus s}$. Consider $C=A$, $D=B\la 1\ra$.
 Then $C\subset H_C=\HH^*_{et}\la-1\ra^{\oplus s}$,
 $D\subset H_D=\HH^*_{et}\la-1\ra^{\oplus s}$ and this pair of Rost submodules satisfies all the conditions of 
 Theorem \ref{curv}(2) (note that for the adjunction
 $$
 \xymatrix{
\dmkF{\zz} \ar@/_0.7pc/ @{->}[r]_{\nu^*} & 
\dmkF{\zz/p} \ar@/_0.7pc/ @{->}[l]_{\nu_*}
},
 $$
 we have: $\Homi_{\dmkF{\zz}}(\nu_*(U),\zz(2)[4])=\nu_*\Homi_{\dmkF{\zz/p}}(U,\zz/p(2)[3])$).
 Thus, it defines a direct summand $N$ in the reduced motive of some smooth projective curve, in particular, a Chow motive
 with $\zz/p$-coefficients. But since $B^j=0$, for $j<0$, it follows that $D^j=0$, for $j\leq 0$. Hence, the Chow motive
 $N^{\vee}$ has no Chow groups $\CH_*$ over any field extension $F/k$. By the standard arguments \cite{Ma68}
 this implies
 that $N^{\vee}=0=N$, and so, $s=0$ and $M=0$.
 \Qed
\end{proof}

Thus, in the Hodge half-diamond of $M$ all 4 potential positions are filled. In particular, we get:

\begin{cor}
\label{Pic-nontriv}
For any non-zero $p$-torsion direct summand $M$ in the motive of a surface (over any field of characteristic zero),
$Pic(M_{\kbar})\neq 0$.
\end{cor}

\begin{proof} 
By the result of S.Gille \cite{Gi14}, the passage to the algebraic closure is conservative for direct summands of motives of surfaces.
The case of an algebraically closed field follows from Proposition \ref{s-t-zero}.
 \Qed
\end{proof}

In the end, let me say few words about endomorphisms of
$p$-torsion motives of surfaces.

The octahedra (\ref{oktaedry}) appear to be functorial.

\begin{prop}
 \label{okt-funct}
 Let $M$ and $N$ be $p$-torsion motives of surfaces. Then any
 map $\ffi:M\row N$ extends uniquely to the map of octahedra
 $(\ref{oktaedry})$.
\end{prop}

\begin{proof}
 Since $M_{\bullet}(A)$ is the 0-th slice of $M$ and $M^{\bullet}(A)$ is the dual to the 0-th slice of $M^{\vee}$, the assignment $M\mapsto M_{\bullet}(A),\,M^{\bullet}(A)$ is functorial. Hence, $\ffi$ extends to a map of upper halves of octahedra (the ones, containing $M$ (respectively, $N$) and $M^{\vee}$ (respectively, $N^{\vee}$)). Since 
 $M_{et}=(\ct_A)_{et}$ and $\Hom(\ct_A(M),\ct_A(N))=
 \Hom(\ct_A(M)_{et},\ct_A(N)_{et})$, the extensions to
 $\ct_A$ and $\ct_A(-1)$ are unique. 
 Since $\Hom(\ct_A(M)(-1)[1],N_{\bullet}(A))=0$, the extension to $M_{\bullet}(A)$ is unique.
 By the dual argument, the same is true about extension to
 $M^{\bullet}(A)$ (as well as to $M_{\bullet}(B)$ and $M^{\bullet}(B)$). Because 
 $\Hom(\ct_A(M)(-1),\ck_A(N))=0$, we obtain the unique
 extension to the $N-E$ triangles in the lower halves of octahedra. Finally, because
 $\Hom(\ct_A(M)(-1),N^{\bullet}(A))=0$, from the fact that our extension commutes with
 $\ck_A(L)\stackrel{[1]}{\low}\ct_A(L)\stackrel{[1]}{\row}
 L^{\bullet}(A)$ (for $L=M,N$), it follows that it commutes
 with $\beta_A$ and, by duality, with $\alpha_A$.
 \Qed
\end{proof}

In particular, we obtain a natural homomorphism
$$
\End(M)\row\End(\ck_A(M))=\Mat(s,t):=\Mat_{s\times s}(\ff_p)\times\Mat_{t\times t}(\ff_p).
$$
Denote the image of it as $R(A)$.

\begin{prop}
 \label{RA-Mat-nilp}
 \begin{itemize}
  \item[$(1)$] $R(A)=\End(M)/Nilp$, where $Nilp$ is finite nilpotent ideal. 
  \item[$(2)$] $\psi\in\Mat(s,t)$ belongs to $R(A)$ if and only if it maps $A$ to itself, if and only if it maps $A^0=\CH_0(M)$ (considered over all generic points of $X$) to itself.
 \end{itemize}
\end{prop}

\begin{proof}
 The kernel of the surjection $\End(M)\twoheadrightarrow R(A)$ consists of those $\ffi$, which are trivial on $M_{\bullet}$, $M^{\bullet}$ and $\ck$ ($A$ and $B$). Then $\ffi$ is nilpotent on $\ct$ and $\ct(-1)$, and so, on $M$. By Proposition \ref{MN-kon} the kernel is finite.
 
 Clearly, any element of $R(A)$ maps $A$ to itself. Conversely, suppose $\psi$ maps $A$ to itself. Then it lifts uniquely to an endomorphism of 
 $M_{\bullet}(A)[-1]\row\ck_A\row 
 M^{\bullet}(A)\row M_{\bullet}(A)$ and $\ffi^{\vee}$ -
 to an endomorphism of a similar $B$-triangle.
 Any $\psi\in\End(\ck_A)$ lifts to an endomorphism
 of $\ck_A\row\ct_A(-1)\stackrel{\cdot\tau}{\row}\ct_A\row\ck_A[1]$.
 Since $\Hom(M_{\bullet}(A),M)=0$, we have:
 $\End(M_{\bullet}(A)\stackrel{[1]}{\row}\ct_A(-1))=
 \End(M_{\bullet}(A)\stackrel{[1]}{\row}\ct_A(-1)\row
 M\row M_{\bullet}(A))$ and, by duality,
 $\End(\ct_A\stackrel{[1]}{\row}M^{\bullet}(A))=
 \End(\ct_A\stackrel{[1]}{\row}M^{\bullet}(A)\row M\row\ct_A)$.
 Finally, because $\Hom(M_{\bullet}(A),\ct_A)=0$, if our extension commutes with $f$ and $(\cdot\tau)$, then it commutes with $g$.
 
 It remains to note that by Proposition \ref{generators-A},
 $A$ is generated by $A^0$ as a Rost cycle module. So, 
 $\psi$ respects $A$ if and only if it respects $A^0$
 (sufficient to know for generic points of $X$).
 \Qed
\end{proof}

\begin{rem}
Observe, that we obtain a canonical splitting 
$R(A)\row\End(M)$. Also, we get a $1$-to-$1$ correspondence
between idempotents of $\End(M)$ and that of $R(A)$.
\end{rem}

\begin{exa}
 \label{End-Godeaux}
 Let $M$ be the Godeaux torsion motive. Then $s=t=1$ and
 $\CH_0(M_{\cc(X)})=\zz/5$ spanned by $(\wt{v},\wt{u})$. Thus,
 $R(A)=\ff_5$ embedded diagonally into 
 $\Mat(1,1)=\ff_5\times\ff_5$. In particular, $M$ is indecomposable (which is also clear from Corollary
 \ref{Pic-nontriv}).
 
 It is easy to see that the ideal $Nilp$, in this case, is
 $1$-dimensional, spanned by the composition
 $M\stackrel{u}{\row}\zz/5(1)[1]\stackrel{Q_0}{\row}
 \zz/5(1)[2]\stackrel{u^{\vee}}{\row} M$. 
 Recall that $M=\nu_*({\mathbf M})$, for some
motive ${\mathbf M}\in\dmkF{\zz/5}$. The same calculations
give: $\End_{\dmkF{\zz/5}}({\mathbf M})=\ff_5$.
\end{exa}

\section{Appendix: Chow motives vis geometric motives}
\label{section-five}

The purpose of this appendix is to show that looking at motivic homology and cohomology of a geometric motive $M$ one
can check, if the object is pure and, moreover, one can determine
the dimension of the respective smooth projective variety
($M$ is a direct summand of) as well as the Tate-shift involved.
The main tool which permits to see it is the weight structure 
of Bondarko \cite{Bon} on $\dmgmk$, and these results can be
alternatively deduced from those of \cite{BK,BS}.
Below, $k$ is any field of characteristic zero.

Recall the standard notations for {\it motivic homology}
$\moho{i}{j}(M,\zz):=\Hom_{DM(k)}(\zz(j)[i],M)$ and 
{\it motivic cohomology} 
$\moco{i}{j}(M,\zz):=\Hom_{DM(k)}(M,\zz(j)[i])$ of the motive $M$.

\begin{prop}
 \label{Chow-crit}
 Suppose $M\in\dmgmk$.
 Then the following conditions are equivalent:
 \begin{itemize}
  \item[$(1)$] $M\in Chow(k)$;
  \item[$(2)$] $\moho{i}{j}(M_K,\zz)=0$, for $i<2j$, while $\moco{i}{j}(M_K,\zz)=0$, for $i>2j$, for any finitely generated field extension $K/k$.
 \end{itemize}
\end{prop}

\begin{proof}
The implication $(1)\Rightarrow (2)$ is clear from the standard properties of motivic cohomology - see \cite{Voe03}.
$(2)\Rightarrow (1)$: As was shown by Bondarko \cite{Bon}, on $\dmgmk$ one has a (unique) bounded non-degenerate weight structure whose heart is the category of Chow-motives $Chow(k)$. Let $\cd^{\leq m}$ and $\cd^{\geq m}$ be the respective subcategories. 

Let us show that $M\in\cd^{\leq 0}$. Indeed, since the weight
structure is bounded, $M\in\cd^{\leq m}$, for some $m$. If $m\leq 0$, we are done. Suppose $m>0$. Then there exists an
exact triangle
$$
\omega_{\leq m-1}M\row M\row\omega_{\geq m}M\row(\omega_{\leq m-1}M)[1]
$$
Since $M\in\cd^{\leq m}$, here $\omega_{\geq m}M=U[m]$, for
some $U\in Chow(k)$. 
Then $\Hom(M,U[m])=\Hom(M\otimes U^{\vee},\zz[m])$ which
is a direct summand in $\Hom(M\otimes M(X),\zz(i)[2i+m])$,
for some smooth projective variety $X$ and some $i\in\zz$.
Since, for any point $y\in X$ of co-dimension $c$, the
group $\Hom(M_{k(y)},\zz(i-c)[2(i-c)+m])$ is zero by our condition, it follows from the Brown-Gersten-Quillen type arguments that $\Hom(M\otimes M(X),\zz(i)[2i+m])=0$.
Thus, the morphism $M\row\omega_{\geq m}M$ is zero, and so,
$M$ is a direct summand of $\omega_{\leq m-1}M$. Hence,
$M\in\cd^{\leq m-1}$. By induction we obtain that
$M\in\cd^{\leq 0}$.

Applying the same arguments to the dual $M^{\vee}$ of $M$
and using triviality of motivic homology of $M$ below the slope $=2$
line, we obtain that $M\in\cd^{\geq 0}$.
Thus, $M$ belongs to the heart $\cd^{\geq 0}\cap\cd^{\leq 0}=
Chow(k)$. 
\phantom{a}\hspace{5mm}
 \Qed
\end{proof}

One can also deduce the above result from 
\cite[Theorem 3.3.3]{BK}.

\begin{prop}
 \label{Chow-r-eff} {\rm (cf. \cite[Theorem 3.2.1(1)]{BS})}
 Suppose, $N\in Chow(k)$ be such a Chow motive that
 $\moho{2i}{i}(N_K,\zz)=0$, for $i<r$ and any finitely generated extension $K/k$. Then there exists
 a smooth projective (possibly reducible) 
 variety $X$ over $k$ such that
 $N$ is a direct summand of $M(X)(r)[2r]$.
\end{prop}

\begin{proof}
 Since $N$, up to Tate-shift, is a direct summand in the motive of some smooth projective variety, which we may assume equi-dimensional (by multiplying the components of it by projective spaces of appropriate dimension), the statement
 follows from the following result.
 \phantom{a}\hspace{5mm} \Qed
 \end{proof}
 
 \begin{lem}
 \label{lem-r-eff}
 If $N$ is a direct summand of $M(Y)$, for some smooth projective equi-dimensional variety $Y$ and
 $\CH_0(N_K)=0$, for any finitely generated extension $K/k$,
 then $N$ is a direct summand of $M(P)(1)[2]$, for some
 smooth projective equi-dimensional variety $P/k$ with
 $\ddim(P)=\ddim(Y)-1$. 
 \end{lem}
 
 \begin{proof}
 Let $Y=\coprod_i Y_i$ be the decomposition of $Y$ into
 connected components and $E_i=k(Y_i)$ be their function
 fields. 
 $N$ is given by some projector $\ffi\in\End_{Corr}(Y)$, represented by some cycle $\Phi\in\CH^*(Y\times Y)$. 
 Since $\CH_0(N)=0$, this projector acts as zero on zero-cycles over any field extension. In particular, it sends the 
 generic point of $Y_i$ over $E_i$ to zero. This means, that
 the restriction of $\Phi_{k(Y_i)}\in\CH^*(Y_{k(Y_i)})$
 is zero. Thus, $\Phi$ is rationally equivalent to some
 cycle supported on $Z\times Y$, where $Z=\coprod_i Z_i$
 and $Z_i\subset Y_i$ is a reduced closed proper subscheme.
 By the results of Hironaka \cite{Hi}, we can resolve the singularities of $Z_i$ by blowing up smooth centers in the 
 singular locus of $Z_i$. Let $R_i$ be the 
 disjoint union of all such centers for $Z_i$, and 
 $\wt{Z}_i\row Z_i$ be the resulting smooth model. We have
 the natural proper map $P_i:=\wt{Z}_i\coprod R_i\row Z_i$. It follows from \cite[Proposition 7.7]{SU} that the respective
 push-forward 
 $\CH_*(P_i\times Y)\row\CH_*(Z_i\times Y)$ is surjective.
 By multiplying the components of $P_i$ by projective
 spaces of appropriate dimension, we may assume that $P$
 is equi-dimensional of dimension one less than $Y$.
 In light of the above surjectivity, the map $\ffi:M(Y)\row M(Y)$ decomposes as 
 $M(Y)\stackrel{\alpha}{\lrow}M(P)(1)[2]
 \stackrel{\beta}{\lrow}M(Y)$. Since $\ffi=\beta\circ\alpha$
 is a projector, so is 
 $\psi=\alpha\circ\beta\circ\alpha\circ\beta$, and the Chow-motive $N=(Y,\ffi)$ is isomorphic to a direct summand
 of $M(P)(1)[2]$. 
\phantom{a}\hspace{5mm}
 \Qed
\end{proof}

\begin{prop}
 \label{N-r-eff}
 Suppose $N\in Chow(k)$ be such a Chow motives that
 $\moho{2i}{i}(N_K,\zz)=0$, for $i<0$, and 
 $\moco{2i}{i}(N_K,\zz)$, for $i>r$, over any
 finitely generated field extension $K/k$. Then
 $N$ is a direct summand in the motive $M(X)$ of some
 smooth projective (possibly reducible) variety $X$ of
 dimension $r$.
\end{prop}

\begin{proof}
 From Proposition \ref{Chow-r-eff}, $N$ is a direct summand
 in the motive $M(Y)$ of some smooth projective (possibly disconnected) equi-dimensional variety $Y/k$. Let $d=\ddim(Y)$. If $d\leq r$, we are done (can always multiply $Y$ by an
 appropriate projective space, if needed). Assume $d>r$.
 Consider $N^{\vee}$ - a direct summand given in $M(Y)$ by
 the dual projector. Then $\moho{2i}{i}(N^{\vee}_K,\zz)=
 \moco{2(d-i)}{d-i}(N_K,\zz)$ because 
 $N^{\vee}=\Homi(N,\zz(d)[2d])$. In particular, 
 $\CH_0(N^{\vee})=0$, for any field
 extension $K/k$ and so, by Lemma \ref{lem-r-eff},
 $N^{\vee}$ is a direct summand of $M(P)(1)[2]$, where
 $P$ is a smooth projective variety of dimension $d-1$.
 Since $M(P)=\Homi(M(P)(1)[2],\zz(d)[2d])$, we get that
 $N$ is a direct summand of $M(P)$. Applying this
 argument inductively, we obtain that $N$ is a direct
 summand of the motive $M(X)$ of a smooth projective
 variety $X$ of dimension $r$.
 \Qed
\end{proof}

This result can be also deduced from 
\cite[Theorem 3.2.1(1), Corollary 2.2.4(2), 
Proposition 5.2.1]{BS}.

\bigskip

\begin{itemize}
\item[address:] {\small School of Mathematical Sciences, University of Nottingham, University Park, Nottingham, NG7 2RD, UK}
\item[email:] {\small\ttfamily alexander.vishik@nottingham.ac.uk}
\end{itemize}

\end{document}